\documentclass[12pt]{article}
\usepackage{amsmath}
\usepackage{latexsym}
\usepackage{amssymb}
\newtheorem{thm}{Theorem}[section]
\newtheorem{la}[thm]{Lemma}
\newtheorem{Defn}[thm]{Definition}
\newtheorem{Remark}[thm]{Remark}
\newtheorem{prop}[thm]{Proposition}
\newtheorem{cor}[thm]{Corollary}
\newtheorem{Example}[thm]{Example}
\newtheorem{Examples}[thm]{Examples}
\newtheorem{Number}[thm]{\!\!}
\newenvironment{defn}{\begin{Defn}\rm}{\end{Defn}}
\newenvironment{example}{\begin{Example}\rm}{\end{Example}}

\newenvironment{rem}{\begin{Remark}\rm}{\end{Remark}}
\newenvironment{numba}{\begin{Number}\rm}{\end{Number}}
\newenvironment{proof}{{\noindent\bf Proof.}}%
                  {\nopagebreak\hspace*{\fill}$\Box$\medskip\medskip\par}   
\newcommand{\Punkt}{\nopagebreak\hspace*{\fill}$\Box$}
\newcommand{\wb}{\overline}
\newcommand{\ve}{\varepsilon}

\newcommand{\wt}{\widetilde}

\newcommand{\isom}{\cong}
\newcommand{\mto}{\mapsto}
\newcommand{\N}{{\mathbb N}}
\newcommand{\bS}{{\mathbb S}}

\newcommand{\R}{{\mathbb R}}
\newcommand{\C}{{\mathbb C}}

\newcommand{\Z}{{\mathbb Z}}

\newcommand{\cE}{{\mathcal E}}
\newcommand{\cC}{{\mathcal C}}
\newcommand{\cW}{{\mathcal W}}
\newcommand{\cK}{{\mathcal K}}
\newcommand{\cV}{{\mathcal V}}
\newcommand{\cO}{{\mathcal O}}
\newcommand{\cS}{{\mathcal S}}
\newcommand{\cF}{{\mathcal F}}
\newcommand{\cT}{{\mathcal T}}

\newcommand{\cg}{{\mathfrak g}}
\newcommand{\ca}{{\mathfrak a}}

\newcommand{\wh}{\widehat}
\newcommand{\one}{{\bf 1}}
\newcommand{\sub}{\subseteq}
\DeclareMathOperator{\GL}{GL}
\DeclareMathOperator{\gU}{U}
\DeclareMathOperator{\gO}{O}
\DeclareMathOperator{\im}{im}

\DeclareMathOperator{\id}{id}

\newcommand{\sbull}{{\scriptscriptstyle \bullet}}

\DeclareMathOperator{\Diff}{Diff}
\DeclareMathOperator{\GermDiff}{GermDiff}

\DeclareMathOperator{\Hol}{Hol}

\DeclareMathOperator{\conv}{conv}
\DeclareMathOperator{\bsd}{bsd}
\DeclareMathOperator{\rk}{rk}
\DeclareMathOperator{\diam}{diam}
\DeclareMathOperator{\Germ}{Germ}

\DeclareMathOperator{\per}{per}
\newcommand{\dl}{{\displaystyle\lim_{\longrightarrow}}}
\newcommand{\pl}{{\displaystyle\lim_{\longleftarrow}}}

\begin{document}
$\;$\\[-23mm]
\begin{center}
{\Large\bf Homotopy groups of ascending unions of\\[2mm]
infinite-dimensional manifolds}\\[4mm]
{\bf Helge Gl\"{o}ckner}\vspace{1.5mm}
\end{center}
\begin{abstract}\vspace{2mm}
\noindent
Let $M$ be a
topological manifold modelled on
topological vector spaces,
which is the union
of an ascending sequence
$M_1\sub M_2\sub\cdots$
of such manifolds.
We formulate a mild
condition ensuring that
$\pi_k(M,p)=\dl\; \pi_k(M_n,p)$\vspace{-.5mm}
for all $k\in \N_0$ and $p\in M$.
This result is useful for Lie theory, because
many important examples of
infinite-dimensional Lie groups
can be expressed as ascending unions of
finite- or infinite-dimensional Lie groups
(whose homotopy groups may be easier to access).
Information on $\pi_0(G)=G/G_0$,
$\pi_1(G)$ and $\pi_2(G)$ is needed to understand
the Lie group extensions
${\bf 1}\to A\to\wh{G}\to G\to{\bf 1}$
of~$G$ with abelian kernels.
The above conclusion remains valid if
$\bigcup_{n\in \N}M_n$ is merely dense in~$M$
(under suitable hypotheses).
Also, ascending unions can be replaced
by (possibly uncountable) directed unions.\vspace{1.4mm}
\end{abstract}
{\footnotesize {\em Classification}:
22E65; 
57N20 
(Primary)
55N65; 
55P10; 
55P42; 
55Q05 (Secondary)
\\[2.5mm]
{\em Key words}:
Infinite-dimensional manifold,
infinite-dimensional Lie group,
directed union, ascending union,
direct limit, inductive limit,
direct limit group,
homotopy group,
path component,
fundamental group,
compact retractivity, compact regularity, direct limit chart,
approximation, weak homotopy type, Palais' theorem, weighted mapping group}\\[6mm]
{\bf Contents}\\[1.7mm]
{\footnotesize
1.\hspace*{3.4mm}Introduction and statement of results\,\dotfill\,\pageref{secintr}\\
2.\hspace*{3.3mm}Preliminaries and notation\,\dotfill\,\pageref{secprel}\\
3.\hspace*{3.3mm}Elementary observations\,\dotfill\,\pageref{sectrivia}\\
4.\hspace*{3.3mm}Technical preparations\,\dotfill\,\pageref{secprp}\\
5.\hspace*{3.3mm}The main result and first consequences\,\dotfill\,\pageref{secmainprf}\\
6.\hspace*{3.3mm}When the inclusion map is a weak homotopy
equivalence\,\dotfill\,\pageref{secpalais}\\
7.\hspace*{3.3mm}Applications to typical Lie groups that are directed unions
of Lie groups\\
\hspace*{6.35mm}or manifolds\,\dotfill\,\pageref{seclie}\\
8.\hspace*{3.3mm}Applications to typical Lie groups that contain a dense union
of Lie groups\,\dotfill\,\pageref{secdnsungp}\\[.5mm]
References\,\dotfill\,\pageref{secref}}\vspace{3.5mm}
\section{Introduction and statement of results}\label{secintr}
A classical result of Palais sheds light on the
homotopy groups of an open subset $U$ of a locally convex
topological vector space~$E$.
He considered a dense vector subspace~$E_\infty$
of~$E$, endowed with the ``finite topology''
(the final topology with respect
to the inclusion maps $F\to E_\infty$
of finite-dimensional vector subspaces~$F$),
and gave $U_\infty:=E_\infty\cap U$
the topology induced by~$E_\infty$.
Then the inclusion map $U_\infty\to U$ is
a weak homotopy equivalence, i.e.,
\[
\pi_k(U,p)\;\isom \; \pi_k(U_\infty,p)\quad\mbox{for each $k\in \N_0$
and $p\in U_\infty$}\vspace{2mm}
\]
(see \cite[Theorem~12]{Pa2}; cf.\ also \cite{Sva} if~$E$ is a Banach space).
Furthermore,
\[
\pi_k(U_\infty,p)\,=\,\dl_{F\in \cF_p}\,\pi_k(U\cap F,p)\,,\vspace{-1.2mm}
\]
where $\cF_p$ is the set of all finite-dimensional
vector subspaces $F\sub E_\infty$ such that $p\in F$
(see, e.g., \cite[Lemma~II.9]{NeH}).\\[2.5mm]
In this article,
we prove certain non-linear variants of these facts,
in situations where linear spaces
have been replaced by topological manifolds (or certain
more general topological spaces).\\[2.5mm]
All topological spaces in this article
are assumed Hausdorff.
Until further notice,
let $M$ be a topological manifold
modelled on (not necessarily locally convex)
topological vector spaces\footnote{Thus $M$ is a Hausdorff topological
space and for each
$p\in M$, there exists an open neighbourhood
$U\sub M$ of $p$,
a topological vector space~$E$
and a homeomorphism $\phi\colon U\to V$ (called a ``chart'')
from~$U$ onto an open subset $V\sub E$.};
we then simply call $M$ a \emph{manifold}.\footnote{Likewise,
(possibly infinite-dimensional) Lie groups modelled on
locally convex spaces (as in \cite{RES}, \cite{GaN} and \cite{NeS}; cf.\ also
\cite{Mil}) will simply be called
``Lie groups.''}
Also, let $(M_\alpha)_{\alpha\in A}$
be an upward directed family
of such manifolds, such that
$M_\infty:=\bigcup_{\alpha\in A}M_\alpha$
is dense in~$M$ and all inclusion
maps $M_\alpha\to M$
and $M_\alpha\to M_\beta$ (for $\alpha\leq \beta$)
are continuous (but not necessarily embeddings).
We describe conditions ensuring that
\begin{equation}\label{desiderat}
\pi_k(M,p)\;=\; \dl_{\alpha\in A_p}\pi_k(M_\alpha,p)
\end{equation}
for all $k\in \N_0$ and $p\in M_\infty$,
where $A_p:=\{\alpha\in A\colon p\in M_\alpha\}$.\\[2.5mm]
If $M=\bigcup_{\alpha\in A}M_\alpha$
and $M$ is
\emph{compactly retractive} in the
sense that each\linebreak
compact set $K\sub M$
is a compact subset of some $M_\alpha$,
then (\ref{desiderat}) is quite obvious
(Proposition~\ref{babcoreg};
compare
\cite[Remark~3.9]{FUN}
and \cite[Lemma~I.1]{NeF}
for special cases,
as well as many works
on homotopy theory
or $K$-theory).\\[2.5mm]
Our goal is to get beyond this limited situation.
To explain our results,
let us assume first that
$M=\bigcup_{\alpha\in A}M_\alpha$.
In this case, we can prove~(\ref{desiderat})
provided that~$M$ \emph{admits weak direct limit charts},
i.e., each point $q\in M$ is contained
in the domain~$U$ of a so-called weak direct limit
chart $\phi\colon U\to V$.
\begin{defn}\label{defweakdl}
\hspace*{-1mm}A \emph{weak direct limit chart}
of $M\!\!=\!\bigcup_{\alpha\in A}\!M_\alpha$
is a chart \mbox{$\phi\colon \!U\!\to\!V$} of~$M$ taking
$U$ homeomorphically onto
an open subset~$V$ of a topological vector space~$E$,
such that there exist $\alpha_0\in A$, charts
$\phi_\alpha\colon U_\alpha\to V_\alpha\sub E_\alpha$
of $M_\alpha$ onto open subsets $V_\alpha\sub E_\alpha$
of topological vector spaces~$E_\alpha$
for $\alpha\geq \alpha_0$,
and continuous linear maps $\lambda_\alpha\colon E_\alpha\to E$
and $\lambda_{\beta,\alpha}\colon E_\alpha\to E_\beta$
(if $\beta\geq\alpha\geq \alpha_0$)
satisfying the following:
\begin{itemize}
\item[(a)]
For all $\alpha\geq\alpha_0$, we have $U_\alpha\sub U$
and $\phi|_{U_\alpha}=\lambda_\alpha\circ\phi_\alpha$;
\item[(b)]
For all $\beta\geq \alpha\geq\alpha_0$, we have
$U_\alpha\sub U_\beta$
and $\phi_\beta|_{U_\alpha}=\lambda_{\beta,\alpha}
\circ\phi_\alpha$;
\item[(c)]
$U=\bigcup_{\alpha\geq\alpha_0}U_\alpha$.
\end{itemize}
\end{defn}
By (a) and (b), each $\lambda_\alpha$
and $\lambda_{\beta,\alpha}$ is injective;
after replacing $E_\alpha$ with $\im(\lambda_\alpha)$
(equipped with the topology making $\lambda_\alpha$
a homeomorphism onto its image),
we may therefore assume henceforth that
$E_\alpha\sub E$ for each $\alpha\geq \alpha_0$,
$E_\alpha\sub E_\beta$ for all $\beta\geq\alpha\geq\alpha_0$,
and that $\lambda_\alpha$ and $\lambda_{\beta,\alpha}$
are the inclusion maps. Then $E=\bigcup_{\alpha\geq \alpha_0}E_\alpha$,
as a consequence of~(c) and~(a).\\[3mm]
Our results comprise:\vspace{-.5mm}
\begin{thm}\label{babythm}
Assume that a manifold $M$ is a directed union
$M=\bigcup_{\alpha\in A}M_\alpha$
of manifolds $M_\alpha$, such that
all inclusion maps $M_\alpha\to M$
and $M_\alpha\to M_\beta$\linebreak
$($for $\alpha\leq \beta)$
are continuous.
If $M$ admits weak direct limit charts, then
\[
\pi_k(M,p)\;=\; \dl_{\alpha\in A_p}\pi_k(M_\alpha,p)
\quad \mbox{for all $\, k\in \N_0$ and $\, p\in M$.}
\]
\end{thm}
\begin{rem}
The concept of a weak
direct limit chart was introduced in~\cite{COM}
in the special case of ascending sequences
$M_1\sub M_2\sub\cdots$ of manifolds
modelled on locally convex spaces.
In these studies,
a certain strengthened concept of
``direct limit chart'' provided
the key to an understanding
of the direct limit properties
of ascending unions $G=\bigcup_{n\in \N}G_n$
of Lie groups
(all prominent examples of which admit
direct limit charts).
Following~\cite{COM},
we might call a weak direct limit chart
a \emph{direct limit chart}
if, moreover, $E$ and each $E_\alpha$
is locally convex and $E=\dl\,E_\alpha$\vspace{-.8mm}
as a locally convex space.
However, this additional property
is irrelevant for our current ends.
\end{rem}
\begin{rem}\label{rem14}
We also mention that
if $M$ and each $M_\alpha$
is a $C^r$-manifold with $r\geq 1$
(in the sense of \cite{RES} or \cite{NeS}),
$\phi\colon U\to V$ and each $\phi_\alpha$
a $C^r$-diffeomorphism,
and $p\in U$,
then there are canonical choices for
the spaces~$E$ and $E_\alpha$,
namely the tangent spaces
$E:=T_p(M)$ and $E_\alpha=T_p(M_\alpha)$.
\end{rem}
\begin{rem}
In contrast to topological
manifolds, the $C^r$-manifolds used in this article
are always assumed to be pure
manifolds (when $r\geq 1$),
i.e., they are modelled on a single
locally convex space.
\end{rem}
If $M_\infty$ is dense,
but not all of~$M$,
then ``well-filled charts''
are an appropriate
substitute
for weak direct limit charts.
The following notation will be useful:
\begin{defn}
If $E$ is a vector space, $Y\sub E$
and $n\in \N$ is fixed, we let $\conv_n(Y)\sub E$
be the set of all convex combinations of
the special form
\[
t_1y_1+\cdots+t_ny_n\,,
\]
where $y_1,\ldots, y_n\in Y$ and $t_1,\ldots, t_n\geq0$
such that $\sum_{j=1}^n t_j=1$.
Thus $\bigcup_{n\in \N}\conv_n(Y)$
is the convex hull $\conv(Y)$ of~$Y$.
Given $X,Y\sub E$, we set
\[
\conv_2(X,Y):=\{tx+(1-t)y\colon x\in X, y\in Y, t\in [0,1]\}\,.
\]
Then\vspace{-5.3mm}
\begin{equation}\label{convindu}
\conv_2(X,\conv_n(X))=\conv_{n+1}(X)\quad\mbox{for all $n\in \N$.}
\end{equation}
\end{defn}
Actually, we can leave the framework of
manifolds, and consider more general
topological spaces (like manifolds with
boundary or manifolds with corners).
The following definition captures exactly what
we need.
\begin{defn}\label{defwlfl}
Let $M$ be a topological space
and $(M_\alpha)_{\alpha\in A}$ be a directed family of
topological spaces such that
$M_\infty:=\bigcup_{\alpha\in A}M_\alpha$
is dense in~$M$ and
all inclusion maps $M_\alpha\to M$ and
$M_\alpha\to M_\beta$ (for $\alpha\leq\beta$)
are continuous.
We say that a homeomorphism
$\phi\colon U\to V\sub E$
from an open subset
$U \sub M$ onto an arbitrary subset~$V$
of a topological vector space~$E$
is a \emph{well-filled chart} of~$M$
if there exist
$\alpha_0\in A$,
homeomorphisms
$\phi_\alpha\colon U_\alpha\to V_\alpha\sub E_\alpha$
from open subsets $U_\alpha\sub M_\alpha$
onto subsets $V_\alpha$ of certain topological
vector spaces $E_\alpha$ for $\alpha\geq\alpha_0$,
and continuous linear maps $\lambda_\alpha\colon E_\alpha\to E$,
$\lambda_{\beta,\alpha}\colon E_\alpha\to E_\beta$
(for $\beta\geq\alpha\geq\alpha_0$)
such that (a) and (b) from Definition~\ref{defweakdl}
hold as well as the following conditions
(d), (e) and (f):
\begin{itemize}
\item[(d)]
$U_\infty:=\bigcup_{\alpha\geq\alpha_0}U_\alpha=U\cap M_\infty$.
\item[(e)]
There exists a non-empty (relatively) open set $V^{(2)}\sub V$
such that $\conv_2(V^{(2)})\sub V$
and $\conv_2(V_\infty^{(2)})\sub V_\infty$,
where $V_\infty:=\bigcup_{\alpha\geq\alpha_0}V_\alpha$
and $V_\infty^{(2)}:=V^{(2)}\cap V_\infty$.
\item[(f)]
For each $\alpha\geq\alpha_0$ and compact set
$K\sub V_\alpha^{(2)}:=V^{(2)}\cap V_\alpha$,
there exists $\beta\geq\alpha$ such that
$\conv_2(K)\sub V_\beta$.
\end{itemize}
Then
$U^{(2)}:=\phi^{-1}(V^{(2)})$ is an open
subset of~$U$, called a \emph{core} of~$\phi$.
For later use, we set
$U^{(2)}_\infty:=\phi^{-1}(V^{(2)}_\infty)$;
then
$U_\alpha^{(2)}:=U^{(2)}_\infty\cap U_\alpha=\phi_\alpha^{-1}(V_\alpha^{(2)})=
\phi_\alpha^{-1}(V^{(2)}\cap V_\alpha)$
is open in $M_\alpha$ and $U^{(2)}_\infty=\bigcup_{\alpha\geq \alpha_0}U^{(2)}_\alpha$.
Also, we abbreviate $E_\infty:=\bigcup_{\alpha\geq \alpha_0}E_\alpha$.\vspace{.2mm}
If cores of well-filled charts
cover~$M$, then $M$ is said to \emph{admit
well-filled charts}.
\end{defn}
\begin{rem}
We hasten to add
that we assumed
in~(e) and~(f)
that
$E_\alpha\sub E$ and that
$\lambda_\alpha$, $\lambda_{\beta,\alpha}$
are the respective inclusion maps
(which we always may as explained after
Definition~\ref{defweakdl}).
\end{rem}
\begin{rem}
Note that $U_\infty$ is dense in~$U$ because~$U$
is open and $M_\infty$ is dense in~$M$.
Consequently, $V_\infty$ is dense in~$V$.
\end{rem}
\begin{rem}
If $V_\alpha$ is open in~$E_\alpha$
for each $\alpha\geq \alpha_0$,
or if each $V_\alpha$ is convex,
then condition~(f) follows from~(e)
and hence can be omitted.
\end{rem}
The reader may find the
concept of a well-filled chart
somewhat elusive. To make it more
tangible, let us consider
relevant special cases:
\begin{example}\label{newd}
If (a), (b) and (d) hold,
$V$ is open in~$E$,
$E_\infty\cap V=V_\infty$
and each $V_\alpha$ is convex or open
in~$E_\alpha$ (to ensure (f)),
then $\phi$ is a well-filled chart.
[In fact, pick any $v\in V$ and
balanced, open $0$-neighbourhood
$W\sub E$ such that $W+W\sub V-v$;
then $V^{(2)}:=v+W$ satisfies~(e).] \,In particular:
\begin{itemize}
\item[(i)]
Every weak direct limit chart
is a well-filled chart.
\item[(ii)]
$\phi$ is a well-filled chart if (a), (b) and (d) hold,
$V$ is open and $E_\alpha\cap \hspace*{.2mm}V =V_\alpha$
for each $\alpha\geq \alpha_0$.
\end{itemize}
\end{example}
\begin{example}\label{newDD}
If (a), (b) and (d) hold,
$V$ is convex,
$E_\infty\cap V=V_\infty$
and each $V_\alpha$ is convex or
open in~$E_\alpha$,
then $\phi$ is a well-filled chart
(with $V^{(2)}:=V$).
\end{example}
We shall obtain
the following far-reaching
generalization of
Theorem~\ref{babythm}.
\begin{thm}\label{manthm}
\hspace*{-.2mm}Consider\hspace*{-.2mm}
a\hspace*{-.2mm} topological\hspace*{-.2mm} space\hspace*{-.1mm} $M$\hspace*{-.72mm}
and\hspace*{-.2mm} a\hspace*{-.2mm} directed\hspace*{-.2mm} family\hspace*{-.2mm}
$(M_\alpha)_{\alpha\in A}$ of
topological spaces whose union
$M_\infty:=\bigcup_{\alpha\in A}M_\alpha$
is dense in~$M$.
Assume that all inclusion maps $M_\alpha\to M$
and $M_\alpha\to M_\beta$ $($for $\alpha\leq\beta)$
are continuous.
If $M$ admits well-filled charts,
then
\[
\pi_k(M,p)\;=\; \dl_{\alpha\in A_p} \,\pi_k(M_\alpha,p)\quad
\mbox{for all $k\in \N_0$ and $p\in M_\infty$.}
\]
\end{thm}
We shall also see that the inclusion map
$M_\infty\to M$ is a weak homotopy equivalence
for suitable topologies on $M_\infty$ (Proposition~\ref{genpalais}).
As a very special case, we obtain a generalization
of Palais' original result:\footnote{Compare \cite[Theorem~13
and end of p.\,1]{Pa2}
for indications of related
generalizations.}
\begin{cor}\label{corpal}
Let $E$ be a topological
vector space $($which need not be locally convex$)$
and $U\sub E$ be a subset such that
\begin{itemize}
\item[\rm(a)]
$U$ is open; or:
\item[\rm(b)]
$U$ is semi-locally convex, i.e.,
each $p\in U$ has a neighbourhood $($relative~$U)$
which is a convex subset of~$E$.
\end{itemize}
Let $E_\infty$ be a vector subspace of~$E$
such that $U_\infty:=U\cap E_\infty$ is dense in~$U$.
Endow $U_\infty$
with the topology~$\cO$
induced by the finite topology on~$E_\infty$.
Then the inclusion map
$(U_\infty,\cO)\to U$
is a weak homotopy equivalence.
Furthermore,
\[
\pi_k(U,p)\;=\; \dl_{F\in\cF_p}\pi_k(U\cap F,p)
\quad \mbox{for each $k\in \N_0$ and $p\in U_\infty$,}\vspace{-1mm}
\]
where $\cF_p$ is the set of finite-dimensional
vector subspaces
$F\!\sub \!E_\infty$ with $p\in F$.
\end{cor}
So far, generalizations to non-locally convex spaces
had been established only for isolated examples~\cite{PaP}.\\[2.5mm]
{\bf Results concerning homotopy classes of general maps.}
Theorem~\ref{manthm}
will be deduced
from an analogous
result (Theorem~\ref{newthm})
for homotopy classes
of continuous maps $|\Sigma|\to M$,
where $\Sigma$ a finite simplicial complex.
This\linebreak
theorem is our main result.
Its proof does not cause additional effort.\\[4mm]
{\bf Applications in Lie theory.}
Once all tools enabling calculations
of\linebreak
homotopy groups
are established (in Section~\ref{secpalais}),
we apply them to typical
examples
of infinite-dimensional
Lie groups.

\noindent
In Section~\ref{seclie},
we inspect the prime examples
of infinite-dimensional
Lie groups
which are directed
unions of Lie groups
or manifolds (as compiled in~\cite{COM}).
As we shall see, our methods
apply to all of them. Many of the examples
are compactly retractive (whence the elementary
Proposition~\ref{babcoreg} applies),
but not all of them (in which case Theorem~\ref{babythm}
cannot be avoided).
It deserves mention that the existence
of a direct limit chart is usually quite obvious,
while the proof
of compact retractivity may require specialized
functional-analytic tools.
Therefore Theorem~\ref{babythm}
is usually easier to apply\linebreak
than Proposition~\ref{babcoreg}
(although its proof is much harder).\\[2.5mm]
The main applications of our results
are given in Section~\ref{secdnsungp},
which is devoted to the
calculation of the homotopy groups
of prime examples of Lie groups
that contain a dense directed union of Lie groups
(notably various types of mapping groups
and diffeomorphism groups).
In particular, we prove a (formerly open) conjecture
by Boseck,
Czichowski and Rudolph~\cite{BCR}
from 1981,
concerning the homotopy groups of Lie groups
of rapidly decreasing\linebreak
Lie-group valued maps on~$\R^d$ (see Remark~\ref{solveBCR}).\\[2.5mm]
As an additional input, our applications
in Section~\ref{secdnsungp}
require that the test function group
$C^\infty_c(M,H)$
is dense in $C^r_c(M,H)$
for each finite-dimensional smooth manifold~$M$,
Lie group~$H$ and $r\in \N_0$.
And a similar density\linebreak
result
is also needed for certain weighted mapping groups.
These more specialized technical tools have been relegated to a separate
paper~\cite{SMO}. They are based on results concerning smooth approximations
of $C^r$-sections in\linebreak
fibre bundles,
which generalize the $C^0$-case discussed in~\cite{Woc}.\\[4mm]
{\bf Motivation.}
In the extension theory of infinite-dimensional
Lie groups,
the homotopy groups
$\pi_0(G)=G/G_0$,
$\pi_1(G)$ and $\pi_2(G)$
are needed to see whether a central extension
\[
\{0\}\to \ca\to \ca \oplus_\omega \cg \to \cg\to \{0\}
\]
of topological Lie algebras
(with $\cg=L(G)$)
gives rise to a central extension
\[
\one \to \ca/\Gamma\to \wh{G}\to G\to \one
\]
of Lie groups
for some discrete subgroup $\Gamma\sub \ca$
and some Lie group $\wh{G}$
such that $L(\wh{G})=\ca \oplus_\omega \cg$
(where $\ca$ is a complete
locally convex space and $\omega$
an $\ca$-valued $2$-cocycle on~$\cg$).
If $G$ is connected (i.e., if $\pi_0(G)=\one$),
then such a Lie group extension exists
if and only if
\begin{itemize}
\item
The ``period group'' $\Pi$ is discrete
in $\ca$ (which is
the image of a certain
``period homomorphism''
$\per_\omega\colon \pi_2(G)\to \ca$); and
\item
A certain ``flux homomorphism''
$F_\omega\colon \pi_1(G)\to H^1_c(\cg, \ca)$ vanishes~\cite{NeC}.
\end{itemize}
In this case, one can take $\Gamma=\Pi$.
Similar results are available for abelian~\cite{NeA}
and general extensions~\cite{NeN}.
In view of these applications,
it is very well motivated to study
the homotopy groups of infinite-dimensional Lie groups.\\[4mm]
{\bf Related literature.}
Much of the literature on the homotopy groups
of infinite-dimensional manifolds
has concentrated on the case of manifolds
modelled on Hilbert, Banach or Fr\'{e}chet
spaces, for which strongest results
are available. We recall two landmark results:
Every smoothly paracompact\linebreak
smooth manifold
modelled on a separable Hilbert space
is diffeomorphic to an open subset
of modelling spaces \cite{EaE},
and its diffeomorphism type is
\linebreak
determined
by the homotopy type \cite{BaK}.
Finite-dimensional submanifolds
play a vital role in~\cite{EaE}.
Frequently, Banach manifolds
are homotopy equivalent
to an ascending union
of finite-dimensional submanifolds
(see \cite{EaE} and \cite{Muk}).\\[2.5mm]
Various authors have studied the homotopy
groups of certain classical
Banach-Lie groups
of operators of Hilbert spaces
(see \cite{Pa1}
and
\cite{dlH}
for the case of\linebreak
separable Hilbert spaces,
\cite{NeH} for discussions subsuming
the non-separable case);
also some results on
groups of operators of Banach spaces
are\linebreak
available~\cite{Geb}.
Typically, one shows that
the group is homotopy equivalent to a direct limit
of classical groups like
$\GL_\infty(\R)=\dl\, \GL_n(\R)$,\vspace{-.5mm}
$\gU_\infty(\C)=\dl\, \gU_n(\C)$
or $\gO_\infty(\R)=\dl\, \gO_n(\R)$.\vspace{-.5mm}
The homotopy groups of these direct limit
groups can be calculated
using the Bott periodicity theorems~\cite{Bot}.
In~\cite{NeF},\linebreak
dense unions of finite-dimensional Lie groups
are used to describe the\linebreak
homotopy groups
of unit groups of approximately finite $C^*$-algebras.\\[2.5mm]
Some results
beyond Banach-Lie groups
are established in \cite{NeG},
notably\linebreak
approximation theorems
enabling the calculation of the homotopy groups
of various types of mapping groups,
like $C_0(M,H)$
with $M$ a $\sigma$-compact finite-dimensional
smooth manifold and~$H$ a Lie group
\cite[Theorem A.10]{NeG}.\\[2.5mm]
Typical applications of direct limits
of finite-dimensional Lie groups (and manifolds)
in algebraic topology are described
in \cite[\S47]{KaM}.
\section{Preliminaries and notation}\label{secprel}
In addition to the definitions already given
in the introduction, we now\linebreak
compile further notation,
conventions and basic facts.\\[2.5mm]
{\bf General conventions.}
As usual, $\R$ denotes the field
of real numbers, $\N:=\{1,2,\ldots\}$, $\N_0:=\N\cup\{0\}$
and $\Z:=\N_0\cup(-\N)$. A subset $U$ of a vector
space~$E$ is called \emph{balanced} if $tU\sub U$
for all $t\in \R$ such that $|t|\leq 1$.
If $(X,d)$ is a metric space,
$x\in X$ and $\ve>0$, we write
$B_\ve^d(x):=\{y\in X\colon d(x,y)<\ve\}$
and
$\wb{B}_\ve^d(x):=\{y\in X\colon d(x,y)\leq
\ve\}$, or simply $B_\ve(x)$ and $\wb{B}_\ve(x)$
if~$X$ and~$d$ are clear from the context.
If $(X,\|.\|)$ is a normed space and $d(x,y)=\|x-y\|$,
we also write $B^X_\ve(x):=B^d_\ve(x)$.
By a directed family, we mean a family
$(X_\alpha)_{\alpha\in A}$ of sets $X_\alpha$
indexed by a directed set $(A,\leq)$
such that $X_\alpha\sub X_\beta$
for all $\alpha,\beta\in A$ such that $\alpha\leq \beta$.
If~$G$ is a topological group,
we write $1$ for its neutral element
and abbreviate $\pi_k(G):=\pi_k(G,1)$
for $k\in \N_0$. If $G_{(1)}$
is the path component of~$1$,
then $\pi_0(G)=G/G_{(1)}$,
whence $\pi_0(G)$ is a group in a natural way.
The following convention is useful:
\begin{numba}\label{concon}
Let $M$ be a topological space, $p\in M$
and $k\in \N_0$.
If $k\geq 1$ or $M$ is a topological
group and $p=1$, then
$\pi_k(M,p)$ is considered as a group,
``morphism'' reads ``homomorphism,''
and we are working in the category of
groups and homomorphisms.
Otherwise, $\pi_0(M)\hspace*{-.2mm}:=\hspace*{-.2mm}\pi_0(M,p)$ is a set,
``morphism'' reads ``map,''
and we are working in the category
of sets and maps.
\end{numba}
Given a map $f\colon X\to Y$ and $A\sub X$,
we write $f|_A$ for the restriction of $f$ to~$A$.
If $B\sub Y$ is a subset which contains
the image $\im(f)$ of~$f$,
we write $f|^B\colon X\to B$
for the co-restriction of~$f$ to~$B$.
Given a topological space~$X$ and $p\in X$,
we let $X_{(p)}$ be the path component of~$p$ in~$X$.
If $f\colon X\to Y$ is a continuous map
and $p\in X$, then $f$ restricts and co-restricts
to a continuous map $f_{(p)}\colon X_{(p)}\to Y_{(f(p))}$.
If $(X,d)$ is a metric space and $A\sub X$
a subset, we let $\diam(A):=\sup\{d(x,y)\colon
x,y\in A\}\in [0,\infty]$ be its diameter.\\[2.5mm]
{\bf Simplicial complexes.}
In this article, we shall only need
finite simplicial complexes~$\Sigma$,
and we shall always consider
these as sets of simplices
$\Delta=\conv\{v_1,\ldots, v_r\}$
in a finite-dimensional vector space~$F$
(where $v_1,\ldots, v_r\in F$ are affinely
independent and $\rk(\Delta):=r$),
not as abstract simplicial\linebreak
complexes.
We write $|\Sigma|:=\bigcup_{\Delta\in \Sigma}\Delta$
and call $\sup\{\rk(\Delta)\colon \Delta\in \Sigma\}\in \N$
the \emph{rank of $\Sigma$.}
Given a simplex $\Delta=\conv\{v_1,\ldots,v_r\}$
as above, we let
$\cV(\Delta):=\{v_1,\ldots, v_r\}$
be its set of vertices.
We define $\cV(\Sigma):=\bigcup_{\Delta\in \Sigma}\cV(\Delta)$.
A simplicial complex $\Sigma'$ is called
a \emph{refinement} of~$\Sigma$
if each $|\Sigma'|=|\Sigma|$
and each $\Delta\in \Sigma$ is a union
of simplices in $\Sigma'$.
A typical example of a refinement
is the barycentric subdivision $\bsd(\Sigma)$
of $\Sigma$ (see \cite[119-120]{Hat}), which we may iterate:
$\bsd^j(\Sigma):=\bsd(\bsd^{j-1}(\Sigma))$ for $j\in \N$.
We recall:
If a euclidean norm $\|.\|$ on~$F$ is given,
$D:=\sup\{\diam(\Delta)\colon \Delta\in \Sigma\}$
and $r:=\rk(\Sigma)$,
then
\begin{equation}\label{estbsd}
\diam(\Delta)\; \leq \; \frac{r-1}{r}\, D
\quad\mbox{for
each $\, \Delta\in \bsd(\Sigma$)}
\end{equation}
(cf.\ \cite[p.\,120]{Hat}).
Triangulating $|\Sigma|$ by affine simplices
ensures that
\begin{equation}\label{goodtri}
\bigcup \, \{\Delta\in\Sigma\colon
\Delta\sub X\}\;=\; X
\end{equation}
if $|\Sigma|$
is a convex set and $X$ a face of~$|\Sigma|$
(or a union of faces).\\[2.5mm]
{\bf Basic facts concerning the compact-open topology.}
If $X$ and $Y$ are topological spaces, we write
$C(X,Y)_{c.o.}$ for the set of
continuous functions from $X$ to~$Y$,
equipped with the compact-open topology.
The sets
\[
\lfloor K, W \rfloor\, :=\, \{\gamma\in C(X,Y)\colon \gamma(K)\sub W\}
\]
form a subbasis for this topology,
for $K$ ranging through the compact
subsets of~$X$ and $W$ through the open subsets
of~$Y$. The following well-known facts
(proved, e.g.,
in \cite{Eng} and \cite{GaN})
will be used repeatedly:
\begin{la}\label{sammeltaxi}
Let $X$, $Y$ and $Z$ be topological
spaces and $f\colon Y\to Z$ be a continuous map.
The the following holds:
\begin{itemize}
\item[\rm(a)]
The map $C(X,f)\colon C(X,Y)_{c.o.}\to C(X,Z)_{c.o.}$,
$\gamma\mto f\circ \gamma$ is continuous.
\item[\rm(b)]
The map $C(f,X)\colon C(Z,X)_{c.o.}\to C(Y,X)_{c.o.}$,
$\gamma\mto \gamma\circ f$ is continuous.\linebreak
In particular,
$C(\iota, X)\colon C(Z,X)\to C(Y,X)$, $\gamma\mto\gamma|_Y$
is continuous if $Y\sub Z$ is equipped with a topology
making the inclusion map $\iota\colon Y\to Z$, $y\mto y$
continuous.
\item[\rm(c)]
If $\gamma\colon X\times Y\to Z$ is continuous,
then $\gamma^\vee\colon X\to C(Y,Z)_{c.o.}$,
$\gamma^\vee(x):=\gamma(x,\sbull)$ is continuous.
\item[\rm(d)]
If $\gamma\colon X\to C(Y,Z)_{c.o.}$ is continuous
and~$Y$ is locally compact, then the map
$\gamma^\wedge\colon X\times Y\to Z$, $\gamma^\wedge(x,y):=\gamma(x)(y)$
is continuous.
\item[\rm(e)]
If $X$ is locally compact, then the evaluation map
\mbox{$\ve\colon \! C(X,Y)_{c.o.}\!\!\times \!X\to Y\!$,}
$\ve(\gamma,x):=\gamma(x)$
is continuous.\Punkt
\end{itemize}
\end{la}
{\bf Direct limits.}
We assume that the reader is familiar
with the concepts of a direct
system
$\cS=((X_\alpha)_{\alpha\in A}, (\phi_{\beta,\alpha})_{\beta\geq \alpha})$
over a directed set $(A,\leq)$
in a category~$\cC$,
the notion of a cone
$(X,(\phi_\alpha)_{\alpha\in A})$
over~$\cS$
and that of a direct limit cone
and its universal property.
It is well known
that every
direct system
$((X_\alpha)_{\alpha\in A}, (\phi_{\beta,\alpha})_{\beta\geq \alpha})$
in the category of
sets has a direct limit
$(X,(\phi_\alpha)_{\alpha\in A})$.
It is also known that
\begin{equation}\label{bscdl1}
X\; =\; \bigcup_{\alpha \in A}\phi_\alpha(X_\alpha)
\end{equation}
and if $\alpha,\beta \in A$ and $x\in X_\alpha$, $y\in X_\beta$, then
\begin{equation}\label{bscdl2}
\phi_\alpha(x)=\phi_\beta(y)\quad
\Leftrightarrow \quad \mbox{$(\exists \gamma \geq \alpha,\beta )$ $\;\;\phi_{\gamma,\alpha}(x)
=\phi_{\gamma,\beta}(y)$.}
\end{equation}
Likewise, each direct system in the category of groups
and homomorphisms has a direct limit. Its underlying set
is the direct limit of the given direct
system in the category of sets. See, e.g.,
\cite[\S2]{DIR} for these well-known facts.
\section{Elementary observations}\label{sectrivia}
In this section,
we make some simple observations
concerning the path\linebreak
components, homotopy groups
and homology modules of directed unions
of topological spaces,
assuming that these are
compactly retractive.
Special cases of these results
are known or part of the folklore,
but they are so useful (and apply to so
many examples in Lie theory) that
they deserve to be recorded
in full generality, despite their simplicity.
First applications
of weak direct limit charts will also be given.
\begin{numba}
Throughout
this section,
we assume that $M$ is a topological space
and $M=\bigcup_{\alpha\in A}M_\alpha$
for a directed family $(M_\alpha)_{\alpha\in A}$
of topological spaces~$M_\alpha$,
such that the inclusion maps
$\lambda_\alpha \colon M_\alpha\to M$ (for $\alpha\in A$)
and $\lambda_{\beta,\alpha}\colon M_\alpha\to M_\beta$ (for $\alpha\leq \beta$)
are continuous.
\end{numba}
\begin{defn}
We say that $M=\bigcup_{\alpha\in A}M_\alpha$
is \emph{compactly retractive}
if every\linebreak
compact subset
$K\sub M$ is contained in
$M_\alpha$ for some $\alpha\in A$
and $M_\alpha$ induces the same
topology on~$K$ as~$M$.
\end{defn}
Compact retractivity has useful consequences:
\begin{prop}\label{babcoreg}
Let $M=\bigcup_{\alpha\in A}M_\alpha$
be compactly retractive, $p\in M$
and $A_p:=\{\alpha\in A\colon p\in M_\alpha\}$.
Then the following holds:
\begin{itemize}
\item[\rm(a)]
The path component~$M_{(p)}$ of~$p$ in~$M$
is the union $M_{(p)}=\bigcup_{\alpha\in A_p} (M_\alpha)_{(p)}$;
\item[\rm(b)]
$\pi_k(M,p)=\dl_{\alpha\in A_p} \pi_k(M_\alpha,p)$\vspace{-.8mm}
as a group, for each $k\in \N$;
\item[\rm(c)]
$\pi_0(M,p)=\dl_{\alpha\in A_p} \pi_0(M_\alpha,p)$\vspace{-1mm}
as a set;
\item[\rm(d)]
If $M$ and each $M_\alpha$ is a topological
group and all $\lambda_\alpha$ and $\lambda_{\beta,\alpha}$
are continuous
homomorphisms, then
$\pi_0(M)=\dl\, \pi_0(M_\alpha)$\vspace{-.8mm}
as a group;
\item[\rm(e)]
The singular homology modules of~$M$ over~$R$ are of the form
$H_k(M,R)=\dl\, H_k(M_\alpha,R)$,\vspace{-.7mm}
for each $k \in \N_0$ and each commutative ring~$R$.
\end{itemize}
\end{prop}
\begin{proof}
(a) By compact retractivity,
every path in~$M$ is a
path in some~$M_\alpha$,
from which the assertion follows.\vspace{1.3mm}

(b), (c) and (d): We shall use the conventions of~{\bf\ref{concon}}.
In the situation of~(d), we let $p:=1$;
in the situation of~(c) and~(d), we let $k:=0$.
We first fix notation
which can be re-used later.
\begin{numba}\label{reudlsit}
After passing to a cofinal subsystem,
we may assume that $p\in M_\alpha$
for each $\alpha\in A$.
Since $\lambda_\beta \circ \lambda_{\beta,\alpha}=\lambda_\alpha$
if $\alpha\leq\beta$,
we have $(\lambda_\beta)_*\circ (\lambda_{\beta,\alpha})_*
=(\lambda_\alpha)_*$,
where $(\lambda_\alpha)_*\colon \pi_k(M_\alpha,p)\to \pi_k(M,p)$
and $(\lambda_{\beta,\alpha})_*\colon \pi_k(M_\alpha,p)\to\pi_k(M_\beta,p)$.
Hence $(\pi_k(M,p),((\lambda_\alpha)_*)_{\alpha\in A})$
is a cone over the direct system
\[
((\pi_k(M_\alpha,p))_{\alpha\in A},
((\lambda_{\beta,\alpha})_*)_{\alpha\leq \beta})\, .
\]
By the universal property of the direct limit,
there exists a unique morphism
$\psi \colon D:=\dl\,\pi_k(M_\alpha,p)\to \pi_k(M,p)$\vspace{-.8mm}
such that $\psi\circ \mu_\alpha=(\lambda_\alpha)_*$
for each $\alpha\in A$, where
$\mu_\alpha\colon \pi_k(M_\alpha,p)\to D$ is the limit map.
\end{numba}
\emph{$\psi$ is surjective.}
To see this, let
$[\gamma]\in \pi_k(M,p)$,
where $\gamma\colon [0,1]^k\to M$
is a continuous map with
$\gamma|_{\partial [0,1]^k}=p$.
By compact retractivity,
$\gamma$ co-restricts to a continuous map
$\eta \colon [0,1]^k\to M_\alpha$
for some $\alpha\in A$.
Then $[\eta]\in \pi_k(M_\alpha,p)$
and $(\lambda_\alpha)_*([\eta])=[\gamma]$.\\[2.5mm]
\emph{$\psi$ is injective.}
To see this, let
$g_1, g_2 \in D$
such that $\psi(g_1)=\psi(g_2)$.
There exist $\alpha\in A$ and $[\gamma_1],[\gamma_2]\in
\pi_k(M_\alpha,p)$ such that
$g_j=\mu_\alpha([\gamma_j])$
for $j\in \{1,2\}$.
Then $\gamma_1,\gamma_2$ are homotopic relative
$\partial [0,1]^k$ in $M$, by means
of the homotopy $F\colon [0,1]^k\times [0,1]\to M$, say.
By compact retractivity, there is $\beta\geq \alpha$ such that
$F$ co-restricts to a continuous map to $M_\beta$.
Then $(\lambda_{\beta,\alpha})_*([\gamma_1])
= (\lambda_{\beta,\alpha})_*([\gamma_2])$
and hence $g_1=\mu_\alpha([\gamma_1])=
\mu_\beta((\lambda_{\beta,\alpha})_* ([\gamma_1]))
= \mu_\beta((\lambda_{\beta,\alpha})_*([\gamma_2]))=g_2$.\vspace{1.3mm}

(e) Let $c=\sum_{\sigma} r_\sigma\, \sigma$
be a singular chain in $M$, where $r_\sigma\in R$
and $F:=\{\sigma\colon r_\sigma\not=0\}$
is finite. Then there exists $\alpha\in A$
such that each $\sigma\in F$ co-restricts
to a continuous map to $M_\alpha$.
Thus $c$ can be considered as a singular chain
in $M_\alpha$.
The assertion now follows
as in the proof of (b).
\end{proof}
In the presence of weak direct limit charts,
compact retractivity can be checked
on the level of modelling spaces.
\begin{prop}\label{inftesiml}
Let $M$ be a topological manifold
which is a directed union
$M=\bigcup_{\alpha\in A}M_\alpha$
of topological manifolds.
\begin{itemize}
\item[\rm(a)]
If $M$ is covered by the domains
of weak direct limit charts\linebreak
$\phi \colon M\supseteq U\to V\sub E$
$($as in Definition~{\rm\ref{defweakdl})}
such that
$E=\bigcup_{\alpha\geq\alpha_0}E_\alpha$
is compactly retractive,
then $M$ is compactly retractive.
\item[\rm(b)]
If $M$ is compactly retractive
and $\phi=\bigcup_{\alpha\geq\alpha_0}\phi_\alpha\colon
M\supseteq U\to V\sub E=\bigcup_{\alpha\geq\alpha_0}E_\alpha$
a weak direct limit chart
with charts $\phi_\alpha\colon M_\alpha\supseteq
U_\alpha\to V_\alpha\subseteq E_\alpha$,
then $E=\bigcup_{\alpha\geq \alpha_0}E_\alpha$
is compactly retractive.
\end{itemize}
\end{prop}
\begin{proof}
(a) Let $K\sub M$ be compact.
Given $x\in K$, let $\phi\colon U\to V$
be a weak direct limit chart as described in (a),
with $x\in U$. There exists a compact neighbourhood
$K_x\sub K\cap U$ of~$x$ in~$K$.
Now the compact retractivity of~$E$
shows that
$\phi(K_x)$ is a compact subset of $E_{\beta_x}$
for some $\beta_x\geq\alpha_0$.
Since $(V_\alpha\cap \phi(K_x))_{\alpha\geq\beta_x}$
is a directed family of sets
and an open cover of the compact set $\phi(K_x)$,
after increasing $\beta_x$ if necessary we may
assume that $\phi(K_x)$ is a compact subset
of $V_{\beta_x}$.
Then $K_x$ is a compact subset
of $U_{\beta_x}$.
There exists a finite subset $F\sub K$ such that
$K=\bigcup_{x\in F}K_x$,
and $\alpha\geq \alpha_0$ such that $\alpha\geq \beta_x$
for all $x\in F$.
Then $K=\bigcup_{x\in F}K_x$
is a compact subset of $M_\alpha$.\vspace{1.3mm}

(b) We may assume that $0\in V$. Let $K\sub E$ be compact.
Given $x\in K$, there exists a compact
neighbourhood $K_x$ of~$0$ in $(K-x)\cap V$.
By compact retractivity of~$M$,
there exists $\beta_x\geq\alpha_0$
such that $\phi^{-1}(K_x)$ is a compact subset
of~$M_{\beta_x}$.
Since $(U_\alpha\cap \phi^{-1}(K_x))_{\alpha\geq\beta_x}$
is a directed family of sets
and an open cover of the compact set $\phi^{-1}(K_x)$,
after increasing $\beta_x$ if necessary we may
assume that $\phi^{-1}(K_x)$ is a compact subset
of $U_{\beta_x}$.
Then $K_x$ is a compact subset
of $V_{\beta_x}$ and hence of~$E_{\beta_x}$.
There exists a finite subset $F\sub K$ such that
$K=\bigcup_{x\in F}(x+K_x)$,
and $\alpha\geq \alpha_0$ such that $F\sub E_\alpha$
and $\alpha\geq \beta_x$
for all $x\in F$.
Then $K=\bigcup_{x\in F}(x+K_x)$
is a compact subset of $E_\alpha$.
\end{proof}
The following corollary
refers to Lie groups modelled
on locally convex spaces, smooth maps and $C^1$-maps
as in \cite{RES}, \cite{GaN}
and~\cite{NeS}
(cf.\ also \cite{Mil}
for the case of sequentially complete modelling
spaces). The tangent space of a Lie group at the
identity element will be denoted by
$L(G):=T_1(G)$.
If $E$ and each $E_\alpha$
is a locally convex space in the definition of a weak direct limit chart
and $\phi$ and each $\phi_\alpha$
is a $C^1$-diffeomorphism,
then we speak of a weak direct limit chart
\emph{of class $C^1$}. Since all translates of a weak direct
limit chart of a Lie group are weak
direct limit charts,
Proposition~\ref{inftesiml} (and Remark~\ref{rem14}) imply:
\begin{cor}\label{noncpreg}
Assume that a Lie group $G$ is a directed
union $G=\bigcup_{\alpha\in A} G_\alpha$
of Lie groups~$G_\alpha$, such that
all inclusion maps are smooth
homomorphisms.
If $G$ admits a weak direct limit chart of class $C^1$
around~$1$, then $G$
is compactly retractive if and only if $L(G)=\bigcup_{\alpha\in A}L(G_\alpha)$
is compactly retractive.\Punkt
\end{cor}
In the case
of Lie groups, other simple hypotheses
lead to conclusions similar to the preceding ones.
We write $G_0$
for the connected component of the identity element~$1$
in a topological group~$G$.
If $G$ is a Lie group, then $G_0=G_{(1)}$
coincides with the path component.
\begin{la}\label{compli}
Consider a Lie group
$G=\bigcup_{\alpha\in A}G_\alpha$
which is a directed union of Lie groups
$($such that each inclusion map
is a smooth homomorphism$)$.
Assume that
$G$ and each $G_\alpha$ has an exponential map,
$L(G)=\bigcup_\alpha L(G_\alpha)$,
and that $\exp_G(L(G))$ is an
identity neighbourhood in~$G$.
Then the identity component of
$G$ is the union $G_0=\bigcup_{\alpha\in A}(G_\alpha)_0$.
\end{la}
\begin{proof}
$S:=\bigcup_{\alpha\in A}(G_\alpha)_0$
is a subgroup of~$G_0$.
Each $v\in L(G)$
belongs to $L(G_\alpha)$
for some~$\alpha$.
Then $\exp_G(v)=\exp_{G_\alpha}(v)\in (G_\alpha)_0\sub S$,
by naturality of~$\exp$.
Thus $\exp_G(L(G))\sub S$, whence
$G_0\sub S$ and therefore $G_0=S$.
\end{proof}
\section{Technical preparations}\label{secprp}
We now prove
several preparatory lemmas,
which will be used in the next section to establish
our main result (Theorem~\ref{newthm}).
The first lemma yields extensions
of continuous maps from the boundary $\partial\Delta$
of a simplex~to~all~of~$\Delta$.\\[2.5mm]
We start with the following setting:
Let $E$ be a topological vector space,
$F$ be a finite-dimensional vector space,
$v_1,\ldots, v_r\in F$ be affinely independent
points, $\Delta:=\conv\{v_1,\ldots, v_r\}$
and $b:=\sum_{j=1}^r\frac{1}{r}v_j$ be the
barycentre of~$\Delta$. We pick (and fix) any
point $x_\Delta \in \partial\Delta$.
To $\gamma\in C(\partial \Delta,E)$, we associate
a function $\Phi(\gamma)\colon \Delta\to E$
as follows:
Given $x\in \Delta$,
there exists a proper face~$X$ of $\Delta$ such that
$x\in \conv(X\cup\{b\})$.
Then $X=\conv(J)$ for a proper subset $J \subset \{v_1,\ldots, v_r\}$
and
\begin{equation}\label{frmx}
x\; =\; tb+\sum_{j\in J}t_jv_j
\end{equation}
with uniquely determined non-negative
real numbers $t$ and $t_j$ for $j\in J$ such that
$t+\sum_{j\in J}t_j=1$. We define
\begin{equation}\label{2conv}
\Phi(\gamma)(x)\; :=\,
\left\{
\begin{array}{cl}
t\gamma(x_\Delta)+(1-t)\gamma\Big(\frac{\sum_{j\in J}t_jv_j}{1-t}\Big) &
\mbox{\,if $\,t<1$;}\\
\gamma(x_\Delta) &\mbox{\,if $\,t=1$.}
\end{array}
\right.
\end{equation}
This definition is independent
of the choice of~$X$,
as follows from the following consideration:
If also $x\in X'=\conv(J')$ and $x=t'b+\sum_{j\in J'}t_j'v_j$,
then $x\in X \cap X'=\conv(J\cap J')$
and thus $t_j=0$ for all $j\in J\setminus J'$
as well as $t_j'=0$ for all $j\in J'\setminus J$.
Now $t_j=t_j'$ for all $j\in J\cap J'$, by uniqueness.
\begin{la}[Filling Lemma]\label{basicfill}
In the preceding situation, we have:
\begin{itemize}
\item[\rm(a)]
For each $\gamma\in C(\partial\Delta,E)$,
the function
$\Phi(\gamma)\colon\Delta\to E$ is
continuous, and $\Phi(\gamma)|_{\partial \Delta}=\gamma$.
\item[\rm(b)]
$\Phi\colon C(\partial\Delta,E)_{c.o.}\to C(\Delta,E)_{c.o.}$,
$\gamma\mto \Phi(\gamma)$
is continuous and linear.
\item[\rm(c)]
If $\gamma$ is constant, taking the value~$y$,
then also $\Phi(\gamma)(x)=y$ for all $x\in \Delta$.
\end{itemize}
\end{la}
\begin{proof}
We first note that $\Phi(\gamma)|_{\partial\Delta}=\gamma$,
by construction.
Next, we claim that the map
\[
\Phi^\wedge\colon C(\partial \Delta, E)\times\Delta\to E\,,\quad
\Phi^\wedge(\gamma,x)\,:=\,\Phi(\gamma)(x)
\]
is continuous. If this is true, then
$\Phi(\gamma)=\Phi^\wedge(\gamma,\sbull)$ is continuous,
proving~(a).
Moreover, $\Phi=(\Phi^\wedge)^\vee$
will be continuous, by Lemma~\ref{sammeltaxi}\,(c).
Since~$\Phi$ is linear by definition,
this gives~(b). Property (c) holds by construction.\\[2.5mm]
The sets $C(\partial \Delta,E)\times \conv(J\cup\{b\})$
form a finite cover of $C(\partial \Delta,E)\times \Delta$
by closed sets,
if $J$ ranges through the proper
subsets of $\{v_1,\ldots, v_r\}$.
Hence~$\Phi^\wedge$ will
be continuous if its restriction to
each set $C(\partial \Delta,E)\times \conv(J\cup\{b\})$ is continuous
(by the Glueing Lemma, \cite[Satz~3.7]{Que}).
To verify this property, let $(\gamma_\alpha,x_\alpha)$
be a convergent net in
$C(\partial \Delta,E)\times \conv(J\cup\{b\})$,
with limit $(\gamma,x)$. We write $x_\alpha=t_\alpha b+\sum_{j\in J}
t_{j,\alpha}v_j$
and $x=tb +\sum_{j\in J}t_jv_j$ as above.
Then $t_\alpha\to t$ and $t_{j,\alpha}\to t_j$.\\[2.5mm]
Case 1: If $t<1$, then $t_\alpha<1$ eventually and
\begin{eqnarray*}
\Phi^\wedge(\gamma_\alpha,x_\alpha)
&= & t_\alpha\,\gamma_\alpha(x_\Delta)+(1-t_\alpha)\,
\gamma_\alpha
\Big(\frac{\sum_{j\in J}t_{j,\alpha}v_j}{1-t_\alpha}\Big)\\
&\to &
t\,\gamma (x_\Delta)+(1-t)\, \gamma
\Big(\frac{\sum_{j\in J}t_jv_j}{1-t}\Big) \, =\, \Phi^\wedge(\gamma,x)\,,
\end{eqnarray*}
exploiting that the evaluation map
$C(\partial \Delta,E)\times \partial \Delta\to E$,
$(\eta,y)\mto \eta(y)$ is continuous
because $\partial \Delta$ is compact (see
Lemma~\ref{sammeltaxi}\,(e)).\\[2.5mm]
Case 2: If $t=1$, then
\[
\Phi^\wedge(\gamma_\alpha, x_\alpha)-\Phi^\wedge(\gamma,x)
\, =\, \gamma_\alpha(x_\Delta)-\gamma(x_\Delta)+R_\alpha
\]
where $R_\alpha=0$ if $t_\alpha=1$ while
\[
R_\alpha=(t_\alpha-1)\gamma_\alpha(x_\Delta)+(1-t_\alpha)
\, \gamma_\alpha\hspace*{-.7mm}\left(\sum_{j\in J}\frac{t_{j,\alpha}v_j}{1-t_\alpha}\right)\vspace{-2.7mm}
\]
if $t_\alpha<1$. Since $\gamma_\alpha(x_\Delta)-\gamma(x_\Delta)\to 0$
by continuity of evaluation (see Lemma~\ref{sammeltaxi}\,(e)),
it only remains to show
that $R_\alpha\to 0$. To verify this, let $U\sub E$ be
a balanced $0$-neighbourhood.
Pick a
balanced open $0$-neighbourhood $V\sub E$ such that $V+V+V+V \sub U$.
Since $\gamma(\partial \Delta)$ is compact and hence
bounded,
there exists $\rho>0$ such that $\gamma(\partial \Delta)
\sub \rho V$.
Then $\gamma(\partial\Delta)+ \rho V$ is
an open neighbourhood of $\gamma(\partial\Delta)$
and hence $\Omega:=
\lfloor \partial \Delta,\gamma(\partial\Delta)+\rho V\rfloor$
is a neighbourhood of~$\gamma$ in $C(\partial\Delta,E)_{c.o.}$.
For~$\alpha$ sufficiently large,
we have $\gamma_\alpha\in \Omega$
and $1-t_\alpha<\rho^{-1}$.
If $t_\alpha=1$,
then $R_\alpha=0\in U$.
If $t_\alpha<1$, then
$R_\alpha \in (1-t_\alpha)(\gamma_\alpha(\partial\Delta)-\gamma_\alpha(x_\Delta))
\subseteq (1-t_\alpha)(\gamma(\partial\Delta)+\rho V-
\gamma(\partial\Delta)-\rho V)
\subseteq (1-t_\alpha)(\rho V+\rho V+\rho V+\rho V)\sub
(1-t_\alpha)\rho U\sub U$
as well. Thus $R_\alpha\to 0$.
\end{proof}
\begin{numba}\label{repeatset}
The next lemmas refer to a setting
already encountered in Theorem~\ref{manthm}:
$M$ is a topological space
and $(M_\alpha)_{\alpha\in A}$ a directed family of
topological spaces whose union
$M_\infty:=\bigcup_{\alpha\in A}M_\alpha$
is dense in~$M$.
We assume that all inclusion maps $M_\alpha\to M$
and $M_\alpha\to M_\beta$ $($for $\alpha\leq\beta)$
are continuous.
Furthermore, we assume that~$M$ admits well-filled charts.
\end{numba}
\begin{la}\label{insUm}
In the situation of {\bf\ref{repeatset}},
let $\alpha\in A$ and
$K\sub M_\alpha$
be a compact set such that
$K \sub U^{(2)}$
for a core $U^{(2)}$ of a well-filled chart $\phi\colon U\to V$
$($as in Definition~{\rm\ref{defwlfl}}$)$.
Then there exists $\beta\geq \alpha_0$
such that $ K \sub U_\beta^{(2)}$.
\end{la}
\begin{proof}
We may assume that $\alpha\geq \alpha_0$, where $\alpha_0$ is as
in Definition~\ref{defwlfl}.
The sets
$K\cap U_\beta^{(2)}$ (for $\beta\geq \alpha$)
form an open cover
of~$K$,\vspace{-.7mm}
and a directed family of sets.
Since $K$ is compact,
there is $\beta\geq \alpha$
such that $K\sub K\cap U_\beta^{(2)}$.
\end{proof}
If $M$ admits well-filled charts,
then there are well-filled
charts with arbitrarily small domain around each point.
More precisely:
\begin{la}\label{core4}
In the situation of~{\bf\ref{repeatset}},
let $q\in M$ and $W$ be a neighbourhood of~$q$ in~$M$.
Then there exists a well-filled chart
$\phi\colon U\to V$ and a core $U^{(2)}$ of~$\phi$
such that $q\in U^{(2)}\sub U\sub W$.
\end{la}
\begin{proof}
By hypothesis, there exists a well-filled
chart
$\wb{\phi}\colon \wb{U}\to\wb{V}\sub E$
and a core $\wb{U}^{(2)}$ thereof such that
$q\in \wb{U}^{(2)}$.
Let $\alpha_0$ and the homeomorphism
$\wb{\phi}_\alpha\colon
M_\alpha\supseteq \wb{U}_\alpha\to \wb{V}_\alpha$ for $\alpha\geq \alpha_0$
be as in
Definition~\ref{defwlfl},
and $\wb{U}_\infty:=\bigcup_{\alpha\geq \alpha_0}\wb{U}_\alpha$.
There exists a balanced, open $0$-neighbourhood
$Q\sub E$ such that $V:=(\wb{\phi}(q)+Q+Q)\cap \wb{V}\sub
\wb{\phi}(W\cap \wb{U})$.
Set
$V^{(2)}:=(\wb{\phi}(q)+Q)\cap \wb{V}^{(2)}\!\!$,
$\, U:=\wb{\phi}^{-1}(V)$,
$U^{(2)}:=\wb{\phi}^{-1}(V^{(2)})$,
$U_\alpha:=\wb{U}_\alpha\cap U$,
$V_\alpha:=\phi(U_\alpha)=\wb{V}_\alpha\cap V$,
$U_\infty:=\bigcup_{\alpha\geq\alpha_0}U_\alpha$
and $\phi_\alpha:=\wb{\phi}_\alpha|_{U_\alpha}^{V_\alpha}$.
Then $q\in U^{(2)}\sub U\sub W$.
Furthermore, $\phi:=\wb{\phi}|_U^V\colon U\to V$
is a well-filled chart. In fact,
(a) and (b) required in Definition~\ref{defwlfl}
hold by construction.
Since $U\cap M_\infty=U\cap\wb{U}\cap M_\infty=
U\cap \wb{U}_\infty=U\cap \bigcup_{\alpha\geq \alpha_0}\wb{U}_\alpha
=\bigcup_{\alpha\geq \alpha_0}U\cap \wb{U}_\alpha
=\bigcup_{\alpha\geq \alpha_0}U_\alpha$,
also (d) holds.
Next, observe that
$\conv_2(V^{(2)})\sub (\wb{\phi}(q)+Q+Q)\cap \wb{V}=V$.
Moreover,
$V_\infty:=\bigcup_{\alpha\geq\alpha_0}V_\alpha=
V\cap \wb{V}_\infty$
and $V^{(2)}_\infty:=V^{(2)}\cap V_\infty$
satisfy
$\conv_2(V^{(2)}_\infty)\sub
V \cap \conv_2(\wb{V}^{(2)}_\infty)
\sub
V \cap \wb{V}_\infty =V_\infty$.
Hence (e) holds.
To verify~(f), let $\alpha\geq\alpha_0$
and $K\sub V^{(2)}_\alpha:=V^{(2)}\cap V_\alpha$
be a compact set. Then $\conv_2(K)\in \wb{V}_\beta$
for some $\beta\geq\alpha$.
Since also $\conv_2(K)\sub \phi(q)+Q+Q$,
we deduce that $\conv_2(K)\sub (\phi(q)+Q+Q)\cap \wb{V}_\beta
=V\cap \wb{V}_\beta = V_\beta$.
\end{proof}
\begin{la}\label{u4v4}
In the situation
of {\bf\ref{repeatset}},
let $q\in M$,
$\phi\colon U\to V$ be a well-filled chart
and $U^{(2)}\sub U$ be a core of~$\phi$ such that
$q\in U^{(2)}$.
Then there exists an open neighbourhood
$U^{(4)}\sub U^{(2)}$
of~$q$ such that $V^{(4)}:=\phi(U^{(4)})$
satisfies $\conv_2(V^{(4)})\sub V^{(2)}$.
\end{la}
\begin{proof}
Apply the construction from
the proof of Lemma~\ref{core4}
to $\wb{\phi}:=\phi$ and $W:=U^{(2)}$.
\end{proof}
The next lemma is the technical backbone of this article.
Given a continuous map
$\gamma_0\colon |\Sigma|\to M$,
it ensures that any $\gamma$ close to~$\gamma_0$
can be approximated by a
continuous map $\eta_\gamma\colon |\Sigma|\to M$
nearby, which has specific additional properties.
In our later applications,
we shall only need
the approximation
$\eta_{\gamma_0}$ to~$\gamma_0$.
However, the inductive proof
makes it necessary to
formulate and prove the lemma
in the stated form.
\begin{la}[Simultaneous Approximations]\label{pivotal}
In the setting of {\bf\ref{repeatset}}, let\linebreak
$\Sigma$
be a finite simplicial complex,
$\gamma_0\colon |\Sigma|\to M$
be a continuous function
and $Q\sub C(|\Sigma|,M)_{c.o.}$
be a neighbourhood of~$\gamma_0$.
Let $\cE\sub |\Sigma|$ be a subset such that
$\cE = \bigcup\, \{\Delta\in \Sigma\colon \Delta\sub \cE\}$.
Then there exist
a finite subset $S \sub |\Sigma|$
containing $\cV(\Sigma)$,
an open neighbourhood
$P$ of $\gamma_0$ in $C(|\Sigma|,M)_{c.o.}$,
and a continuous map
$\Theta\colon P \times |\Sigma|\times [0,1]\to M$
with the following properties:
\begin{itemize}
\item[\rm(a)]
$\Theta(\gamma,\sbull,0)=\gamma$,
for each $\gamma\in P$;
\item[\rm(b)]
$\Theta(\gamma,\sbull,t)\in Q$,
for each $\gamma\in P$
and $t\in [0,1]$;
\item[\rm(c)]
For each $\gamma\in P$,
the map
$\eta_\gamma:=\Theta(\gamma,\sbull,1)\colon |\Sigma|\to M$
only depends\linebreak
on $\gamma|_{S\cup\cE}$.
Also,
for each $\Delta\in \Sigma$,
the restriction $\eta_\gamma|_\Delta$
only depends on $\gamma|_{(S\cup\cE)\cap\Delta}$;
\item[\rm(d)]
Let $\gamma\in P$
such that $\gamma(S\cup \cE)\sub M_\alpha$
for some $\alpha\in A$,
and $\gamma|_{\cE}\colon \cE\to M_\alpha$
is continuous.
Then there exists $\beta \geq \alpha$ such that
$\eta_\gamma$ takes its values in $M_\beta$
and is continuous as a map to~$M_\beta$;
\item[\rm(e)]
$F_\gamma:=\Theta(\gamma,\sbull)\colon |\Sigma|\times [0,1]\to M$
is a homotopy from $\gamma$ to $\eta_\gamma$,
for each $\gamma\in P$;
\item[\rm(f)]
If $\gamma\in P$ is
such that $\im(\gamma)\sub M_\alpha$
for some $\alpha\in A$
and \mbox{$\gamma|^{M_\alpha} \colon |\Sigma|\to M_\alpha$}
is continuous,
then there exists $\beta \geq \alpha$ such that
$\im(F_\gamma)\sub M_\beta$
and $F_\gamma\colon |\Sigma|\times [0,1]\to M_\beta$
is continuous;
\item[\rm(g)]
If $\gamma\in P$ and $\Delta\in \Sigma$ are
such that $\gamma|_\Delta$ is a constant function,
taking the value $y\in M$, say,
then $F_\gamma(x,t)=y$ for all $x\in \Delta$
and $t\in [0,1]$;
\item[\rm(h)]
$F_\gamma(x,t)=\gamma(x)$ for all $\gamma\in P$,
$x\in S\cup\cE$ and $t\in [0,1]$.
\end{itemize}
\end{la}
\begin{proof}
The proof is by induction on the rank~$r$ of~$\Sigma$.
We assume that $r=1$ first,
in which case $S:=|\Sigma|$
is a finite subset
of a finite-dimensional
vector space~$F$.
Then
$P:=Q$
and $\Theta\colon P \times |\Sigma|\times [0,1]\to M$,
$\Theta(\gamma,x,t):=\gamma(x)$
have the asserted properties.\\[2.5mm]
To perform the induction step,
let $\Sigma$ be a simplicial
complex of rank $r\geq 2$
and assume that the assertion
holds for complexes of rank $r-1$.\\[2.5mm]
By definition of the compact-open topology,
there exist $\ell\in \N$, compact subsets $K_j\sub |\Sigma|$
for $j\in \{1,\ldots, \ell\}$
and open sets $W_j\sub M$ such that
\begin{equation}\label{pregood}
\gamma_0\; \in \; \bigcap_{j=1}^\ell \; \lfloor K_j,W_j\rfloor
\; \sub \; Q\,.
\end{equation}
Our first objective is to make more intelligent choices
of the sets $K_j$ and $W_j$. We shall improve them in
several steps.\\[2.5mm]
Since $\lfloor K_j,W_j\rfloor=\bigcap_{\Delta\in \Sigma}
\lfloor K_j\cap \Delta, W_j\rfloor$,
we may assume without loss of generality
that each $K_j$ is a subset of
some $\Delta_j\in \Sigma$.
Since $\gamma_0^{-1}(W_j)\cap \Delta_j$
is an open neighbourhood of $K_j$ in $\Delta_j$,
there exists $m_j\in \N$ such that
$\gamma_0(\Delta')\sub W_j$
for all $\Delta'\in \bsd^{m_j}(\Delta_j)$
such that $\Delta'\cap K_j\not=\emptyset$ (cf.\ (\ref{estbsd})).
Let $m$ be the maximum of the $m_j$
for $j\in \{1,\ldots, \ell\}$.
After replacing
$K_j$ by all $\Delta'\in \bsd^m(\Delta_j)$
such that $\Delta'\cap K_j\not=\emptyset$,
and after replacing $\Sigma$ with $\bsd^m(\Sigma)$,
we may assume without loss
of generality that $K_j\in \Sigma$ for each $j$.
Given $\Delta\in \Sigma$, define
$W_\Delta$ as the intersection of
the $W_j$, for all $j\in \{1,\ldots,\ell\}$
such that $K_j=\Delta$
(with the convention that
$\bigcap\emptyset:=M$).
Improving (\ref{pregood}),
we now have
\[
\gamma_0\;\in\; \bigcap_{\Delta\in \Sigma}\,
\lfloor \Delta, W_\Delta\rfloor
\;\sub\; Q\,.
\]
In the next step, we replace some $W_\Delta$
by cores of well-filled charts.\\[2.5mm]
Recall that $|\Sigma|\sub F$ for
some finite-dimensional vector space~$F$;
we choose any norm $\|.\|$ on~$F$
and let $d$ be the metric on~$|\Sigma|$
arising from~$\|.\|$.
Given $\Delta'\in \Sigma$
and $x\in \Delta'$,
there exists a well-filled
chart $\phi_{\Delta',x}\colon U_{\Delta',x}\to V_{\Delta',x}$ of~$M$
such that $U_{\Delta',x}\sub W_{\Delta'}$
and $\gamma_0(x)\in U^{(2)}_{\Delta', x}$
for some core $U^{(2)}_{\Delta', x}$ of~$\phi_{\Delta', x}$,
by Lemma~\ref{core4}.
Let $U^{(4)}_{\Delta', x}\sub U^{(2)}_{\Delta', x}$
be a neighbourhood of $\gamma_0(x)$
as in Lemma~\ref{u4v4}.
Since $\gamma_0$ is continuous, $x$ has an open
neighbourhood $Y_{\Delta', x}$ in $\Delta'$
such that $\gamma_0(Y_{\Delta' , x})\sub U^{(4)}_{\Delta', x}$.
Choose $\delta>0$ such that~$\delta$
is a Lebesgue number
for the open cover $(Y_{\Delta', x})_{x\in \Delta'}$
of $\Delta'$,
for each $\Delta'\in \Sigma$.
There exists $m\in \N$
such that $\diam(\Delta)<\delta$
for each $\Delta\in \bsd^m(\Sigma)$ (cf.\ (\ref{estbsd})).
Given $\Delta\in \bsd^m(\Sigma)$,
there exists a unique $\Delta'\in \Sigma$
such that $\Delta\in \bsd^m(\Delta')$
but $\Delta\not\in \bsd^m(\Delta'')$
for each proper face $\Delta''$ of~$\Delta'$.
We
pick $x\in \Delta'$ such that
$\Delta\sub Y_{\Delta', x}$
and set $\phi_\Delta:=\phi_{\Delta',x}$, $U_\Delta:=U_{\Delta',x}$,
$V_\Delta:=V_{\Delta', x}$, $U^{(2)}_\Delta:=U^{(2)}_{\Delta',x}$
and $U^{(4)}_\Delta:=U^{(4)}_{\Delta', x}$.
We let $E_\Delta$ be the topological vector space with $V_\Delta\sub E_\Delta$.
Then $\gamma_0\in \lfloor \Delta, U^{(4)}_\Delta\rfloor$.
As before, replace $\Sigma$ with
$\bsd^m(\Sigma)$ for simplicity of notation.
Let $\Sigma^*$ be the simplicial
complex formed by all simplices
$\Delta\in \Sigma$ of rank at most
$r-1$.
Given $\Delta'\in \Sigma^*$,
let $Z_{\Delta'}:=\bigcap_\Delta U_\Delta^{(4)}$,
where $\Delta$ ranges through all
$\Delta\in \Sigma$ such that $\Delta'\sub \Delta$.
We have achieved the following:
\begin{itemize}
\item[(i)]
$U_\Delta$ is the domain of a well-filled chart
of~$M$, for each simplex
$\Delta\in \Sigma$ of rank~$r$;
\item[(ii)]
If $\Delta,\Delta' \in \Sigma$
such that $\rk(\Delta)=r$
and $\Delta'$ is a proper subset of $\Delta$,
then $Z_{\Delta'} \sub U_\Delta^{(4)}$;
\item[(iii)]
$\gamma_0\in
R \, :=\,
\bigcap_{\Delta\in \Sigma^*}
\lfloor \Delta,Z_\Delta \rfloor\cap
\bigcap_{\Delta\in \Sigma\setminus \Sigma^*}
\lfloor \Delta,U_\Delta^{(4)}\rfloor
\,\sub\,
\bigcap_{\Delta\in \Sigma }\lfloor \Delta,U_\Delta \rfloor
\sub Q$.
\end{itemize}
Define
\[
Q^*\; :=\; 
\bigcap_{\Delta \in \Sigma^*}\lfloor \Delta ,Z_\Delta \rfloor\;\sub\;
C(|\Sigma^*|,M)\,.
\]
By induction, there exists an
open neighbourhood $P^*\sub C(|\Sigma^*|,M)_{c.o.}$
of $\gamma_0|_{|\Sigma^*|}$,
a continuous map $\Theta^*\colon
P^* \times |\Sigma^*|\times [0,1]\to M$
and a finite subset $S\sub |\Sigma^*|$
with $\cV(\Sigma)=\cV(\Sigma^*)\sub S$
satisfying analogues of (a)--(h),
with $\Sigma$ replaced
by $\Sigma^*$, $P$ by $P^*$, $Q$ by~$Q^*$, $\Theta$ by~$\Theta^*$,
and $\cE$ by $\cE^*:=\cE\cap|\Sigma^*|$.
We let
\[
P\;:=\; \{\gamma\in R \colon \gamma|_{|\Sigma^*|}\in P^*\}\,;
\]
by Lemma~\ref{sammeltaxi}\,(b),
this is an open neighbourhood
of $\gamma_0$ in $C(|\Sigma|,M)_{c.o.}$.\\[2mm]
To enable a piecewise definition of~$\Theta$,
let $\Delta\in \Sigma$ be a simplex
of rank~$r$ (which we fix for the moment).
The well-filled
chart $\phi_\Delta \colon U_\Delta\to V_\Delta\sub E_\Delta=:E$
from above goes along with
$\alpha_0\in A$, homeomorphisms $\phi_{\Delta,\alpha}\colon M_\alpha
\supseteq U_{\Delta,\alpha}\to V_{\Delta,\alpha}\sub E_{\Delta,\alpha}$
and sets
$U_{\Delta,\infty}$, $V_{\Delta,\infty}$, $E_{\Delta,\infty}\sub E$,
$V^{(2)}_\Delta$, $V^{(2)}_{\Delta,\infty}$
(etc.) as in Definition~\ref{defwlfl}.
Then
\begin{equation}\label{use3}
\Theta^*(\gamma,x,t)\in U^{(4)}_\Delta\quad
\mbox{for all $\gamma\in P^*$, $x\in\partial\Delta$ and $t\in [0,1]$.}
\end{equation}
In fact, given $x\in \partial\Delta$, there exists
a proper face $\Delta'$ of~$\Delta$
such that $x\in \Delta'$.
Now $\Theta^*(\gamma,x,t)\in Z_{\Delta'}\sub U^{(4)}_\Delta$,
by definition of~$Q^*$ and~$Z_{\Delta'}$.
The preceding enables us to define a map
$\Xi_\Delta\colon P^*\times [0,1]
\to C(\partial\Delta,E)$
via
\[
\Xi_\Delta(\gamma,t):=\phi_\Delta \circ
\Theta^*(\gamma,\sbull,t)|_{\partial\Delta}\,.
\]
As a consequence of (\ref{use3}),
the map~$\Xi_\Delta$
has image in $C(\partial\Delta, V^{(4)}_\Delta)$;
and by Lemma~\ref{sammeltaxi}\,(a)--(c),
$\Xi_\Delta$ is continuous.
We pick $x_\Delta \in S\cap\partial\Delta$
and let
\[
\Phi_\Delta\,:=\,\Phi\colon C(\partial \Delta,E)\to C(\Delta,E)
\]
be as in Lemma~\ref{basicfill}.
By (\ref{2conv}), the values of
$\Phi_\Delta(\gamma)$
lie in
$\conv_2(\im(\gamma))$.
Hence
\begin{equation}\label{enab}
\Phi_\Delta(\gamma)(\Delta)\,\sub\, V_\Delta^{(2)}\quad
\mbox{for each $\gamma\in \lfloor \partial \Delta,V_\Delta^{(4)}\rfloor
\,\sub\, C(\partial \Delta,E)$.}
\end{equation}
Because $\gamma(x)\in U_\Delta^{(4)}$
and thus $\phi_\Delta(\gamma(x))\in V_\Delta^{(4)}$
for all $\gamma\in P\sub R$
and $x\in \Delta$,
we can define
a map $\Theta_\Delta\colon P\times \Delta\times[0,1]\to U_\Delta\sub M$
via
\[
\Theta_\Delta(\gamma,x,t) :=
\left\{
\begin{array}{cl}
\phi^{-1}_\Delta\big((1-2t)\,\phi_\Delta(\gamma(x))+2t\,
\Phi_\Delta(\phi_\Delta\circ \gamma|_{\partial\Delta})(x)\big) &\,
\mbox{if $\,t\in [0,\frac{1}{2}]$;}\\[1mm]
\phi^{-1}_\Delta\big(\Phi_\Delta(\Xi_\Delta(\gamma|_{|\Sigma^*|},2t-1))(x)\big)
&
\,\mbox{if $\,t\in [\frac{1}{2},1]$.}
\end{array}
\right.
\]
This map is continuous as a consequence
of Lemma~\ref{sammeltaxi} (a), (b) and (e).
We now define a map
$\Theta\colon P\times|\Sigma|\times [0,1]\to M$,
as follows:
If $x\in \Delta$
for some $\Delta\in \Sigma$ such that $\rk(\Delta)=r$
and $\Delta\not\sub \cE$,
we set
\[
\Theta(\gamma,x,t)\, :=\,
\Theta_\Delta(\gamma,x,t)\,.
\]
If $x\in \Delta$ for some $\Delta\in \Sigma$
such that $\rk(\Delta)=r$ and $\Delta\sub \cE$,
we set
\[
\Theta(\gamma,x,t)\; := \; \gamma(x)\,.
\]
If $x\in |\Sigma^*|$, we define
\[
\Theta(\gamma,x,t)\, :=\,
\left\{
\begin{array}{cl}
\gamma(x) & \,\mbox{if $\,t\in [0,\frac{1}{2}]$;}\\[1mm]
\Theta^*(\gamma|_{|\Sigma^*|},x,2t-1) &\,\mbox{if $\,t
\in [\frac{1}{2},1]$.}
\end{array}
\right.
\]
If $\Delta\in \Sigma$ is a simplex of rank~$r$,
then $\Theta_\Delta(\gamma,x,t)=\gamma(x)$
for all $t\in [0,\frac{1}{2}]$
and $x\in \partial \Delta$.
Therefore $\Theta$ is well defined.
By the Glueing Lemma, $\Theta$ is continuous.
It remains to show that $\Theta$ satisfies
all of (a)--(h).\vspace{1.3mm}

(a) Let $\gamma\in P$ and $x\in |\Sigma|$.
If $x\in |\Sigma^*|$, then $\Theta(\gamma,x,0)=\gamma(x)$
by definition of~$\Theta$.
Otherwise, $x\in\Delta$ for some $\Delta\in \Sigma$
of rank~$r$. If $\Delta\sub \cE$, then $\Theta(\gamma,x,0)=\gamma(x)$
by definition of~$\Theta$. If $\Delta\not\sub \cE$,
then
$\Theta(\gamma,x,0)=\Theta_\Delta(\gamma,x,0)=\gamma(x)$
by definition of $\Theta_\Delta$.\vspace{1.3mm}

(b) Let $\gamma\in P$ and $t\in [0,\frac{1}{2}]$.
Let $\Delta\in \Sigma$.
If $\rk(\Delta)<r$,
then $x\in|\Sigma^*|$ for each $x\in\Delta$
and thus $\Theta(\gamma,x,t)=
\gamma(x)\in Z_\Delta\sub U_\Delta$
(since $P\sub R$), i.e.,
\begin{equation}\label{rpt1}
\Theta(\gamma,\sbull,t) \in\lfloor \Delta,U_\Delta\rfloor\, .
\end{equation}
Now assume that $\Delta$ has rank~$r$.
If $\Delta\sub \cE$,
then $\Theta(\gamma,x,t)=
\gamma(x)\in U_\Delta$.
If $\Delta\not\sub \cE$,
then $\Theta(\gamma,x,t)=
\Theta_\Delta(\gamma,x,t)\in U_\Delta$
for each $x\in \Delta$,
whence again (\ref{rpt1}) holds.
Thus
\begin{equation}\label{rept}
\Theta(\gamma,\sbull, t) \; \in \; \bigcap_{\Delta\in \Sigma}\,
\lfloor\Delta,U_\Delta\rfloor
\;\sub \; Q\,.
\end{equation}
Now let $\gamma\in P$ and $t\in [\frac{1}{2},1]$.
Let $\Delta\in \Sigma$. If $\rk(\Delta)=r$,
we see as before that (\ref{rpt1}) holds.
If $\rk(\Delta)<r$, we exploit that
$\gamma|_{|\Sigma^*|}\in P^*$ by definition of~$P$,
whence $\Theta^*(\gamma|_{|\Sigma^*|},\sbull ,2t-1)\in Q^*\sub
\lfloor \Delta, U_\Delta\rfloor\sub C(|\Sigma^*|,M)$.
Then
$\Theta(\gamma,x,t)=\Theta^*(\gamma|_{|\Sigma^*|},x,2t-1)\in U_\Delta$
for each $x\in \Delta$ and hence (\ref{rpt1}) holds.
Thus (\ref{rept}) is established.\vspace{1.3mm}

(c) It suffices to prove the second
assertion. To this end, let $\Delta\in \Sigma$ and $x\in \Delta$.
If $\Delta\in \Sigma^*$, we have
$\Theta(\gamma,x,1)=\Theta^*(\gamma|_{|\Sigma^*|},x,1)$,
which only depends on $\gamma|_{(S\cup \cE^*)\cap \Delta}$ by the inductive hypothesis.
If $\rk(\Delta)=r$ and $\Delta\not\sub \cE$,
then
$\Xi_\Delta(\gamma|_{|\Sigma^*},1)
=\phi_\Delta\circ\Theta^*(\gamma|_{|\Sigma^*|},\sbull,1)|_{\partial \Delta}$
only depends on $\gamma|_{(S\cup \cE^*)\cap \Delta}$
(because $\partial\Delta$ is
a union of proper faces $\Delta'$ of $\Delta$,
and $\Theta^*(\gamma|_{|\Sigma^*|},\sbull,1)|_{\Delta'}$
only depends on $\gamma|_{(S\cup \cE^*)\cap \Delta'}$).
Hence also
$\Theta(\gamma,x,1)=\Theta_\Delta(\gamma,x,1)=
\phi_\Delta^{-1}(\Phi_\Delta(\Xi_\Delta(\gamma|_{|\Sigma^*|},1))(x))$
only depends on $\gamma|_{(S\cup\cE^*)\cap\Delta}$.
Finally, assume $\rk(\Delta)=r$ and $\Delta\sub \cE$.
Then $\eta_\gamma|_\Delta=\gamma|_\Delta=\gamma|_{\cE\cap \Delta}$
only depends on $\gamma|_{(S\cup\cE)\cap\Delta}$.\vspace{1.3mm}

(d) It suffices to show that for
each $\Delta\in\Sigma$, there exists $\beta\geq \alpha$
such that $\eta_\gamma(\Delta)\sub M_\beta$ and
$\eta_\gamma|_\Delta^{M_\beta}\colon \Delta\to M_\beta$
is continuous.
If $\rk(\Delta)=r$ and $\Delta\sub \cE$,
then the latter holds by hypothesis (with $\beta:=\alpha$),
since $\eta_\gamma|_\Delta=\gamma|_\Delta$.
To tackle the remaining cases,
we exploit that there is $\tau\geq \alpha$
such that $\Theta^*(\gamma|_{|\Sigma^*|},|\Sigma^*|,1)\sub M_\tau$
and $\Theta^*(\gamma|_{|\Sigma^*|},\sbull ,1)|^{M_\tau}\colon
|\Sigma^*|\to M_\tau$ is continuous,
by the inductive hypothesis.\\[2.5mm]
If $\rk(\Delta)<r$, then
$\eta_\gamma(\Delta)\sub
\Theta^*(\gamma|_{|\Sigma^*|},|\Sigma^*|,1)\sub M_\tau$
and $\Theta^*(\gamma|_{|\Sigma^*|},\sbull ,1)|^{M_\tau}_\Delta$
$=\eta_\gamma|_\Delta^{M_\tau}$
is continuous, whence $\beta:=\tau$
satisfies our needs.\\[2.5mm]
Now assume $\rk(\Delta)=r$ and $\Delta\not\sub \cE$.
Since
$\Theta^*(\gamma|_{|\Sigma^*|},\sbull ,1)|_{\partial\Delta}$\vspace{-.5mm}
has
image in $U_\Delta^{(2)}\cap M_\tau$
and is continuous as a map to~$M_\tau$,
using Lemma~\ref{insUm} we find $\sigma \geq\tau$
such that\vspace{-.3mm}
$\Theta^*(\gamma|_{|\Sigma^*|},\partial\Delta,1)\sub U_{\Delta,\sigma}^{(2)}$
and $\Theta^*(\gamma|_{|\Sigma^*|},
\sbull ,1)|_{\partial\Delta}$
is continuous\vspace{-.3mm}
as a map to $U_{\Delta,\sigma}^{(2)}$.
As a consequence,
$\Xi_\Delta(\gamma|_{|\Sigma^*|},1)\in
C(\partial \Delta, V_{\Delta,\sigma}^{(2)})$.
Now $\Phi_\Delta$, restricted to
$C(\partial\Delta, E_{\Delta,\sigma})$, is
a map to $C(\Delta,E_{\Delta,\sigma})$
by Lemma~\ref{basicfill}\,(a) (applied with $E_{\Delta,\sigma}$
rather than $E=E_\Delta$).
Hence $\phi_\Delta\circ \eta_\gamma|_\Delta
=\Phi_\Delta(\Xi_\Delta(\gamma|_{|\Sigma^*|},1))
\in C(\Delta, E_{\Delta,\sigma})$.
The image~$K$ of $\Xi_\Delta(\gamma|_{|\Sigma^*|},1)$
is a compact subset of $V_{\Delta,\sigma}^{(2)}$.
Hence, by Definition~\ref{defwlfl}\,(f),
there exists $\beta\geq \sigma$ such that $\conv_2(K)\sub V_{\Delta,\beta}$.
As a consequence,
$\eta_\gamma(x)=\phi^{-1}_{\Delta,\beta}
\big(\Phi_\Delta(\Xi_\Delta(\gamma|_{|\Sigma^*|},1))(x)\big)
\in U_{\Delta,\beta}$ for each $x\in \Delta$
and $\eta_\gamma|_\Delta$
is continuous as a map to~$U_{\Delta,\beta}$.\vspace{1.3mm}

(e) $F_\gamma$ is continuous and
hence is a homotopy from
$F_\gamma(\sbull,0)=\gamma$ (see (a))
to $F_\gamma(\sbull,1)=\eta_\gamma$.\vspace{1.3mm}

(f) Define $F_\gamma^*
:=\Theta^*(\gamma|_{|\Sigma^*|},\sbull)\colon |\Sigma^*|\times [0,1]\to M$.
It suffices to show that
for each
$\Delta\in \Sigma$, there exists $\beta\geq \alpha$
such that $F_\gamma(\Delta\times [0,1])\sub M_\beta$
and $F_\gamma|_{\Delta\times [0,1]}\colon
\Delta\times [0,1]\to M_\beta$ is continuous.
If $\rk(\Delta)=r$ and $\Delta\sub \cE$,
then $F_\gamma(x,t)=\gamma(x)\in M_\alpha$;
since $\gamma|^{M_\alpha}$ is continuous
by hypothesis, the
desired property is satisfied with $\beta:=\alpha$.
To tackle the remaining cases,
we shall exploit that there exists $\tau\geq \alpha$
such that $F^*_\gamma(|\Sigma^*|\times [0,1])\sub M_\tau$
and $F^*_\gamma$ is continuous
as a map to $M_\tau$,
by the inductive hypothesis.\\[2mm]
If $\rk(\Delta)<r$, then
$F_\gamma(\Delta\times [0,1])\sub
F^*_\gamma(|\Sigma^*|\times [0,1])
\sub M_\tau$ holds
and $F_\gamma|_{\Delta\times [0,1]}^{M_\tau}$ is continuous
as $F_\gamma(\sbull,t)|_\Delta=\gamma|_\Delta^{M_\tau}$
if $t\in [0,\frac{1}{2}]$
and
$F_\gamma(\sbull,t)|_\Delta=F^*_\gamma(\sbull,2t-1)|^{M_\tau}_\Delta$
if $t\in [\frac{1}{2},1]$.
Hence $\beta:=\tau$
does the job.\\[2mm]
Now assume $\rk(\Delta)=r$ and $\Delta\not\sub \cE$.
Since\vspace{-.4mm}
$F^*_\gamma|_{\partial\Delta\times [0,1]}=
\Theta^*(\gamma|_{|\Sigma^*|},\sbull )|_{\partial\Delta\times [0,1]}$
has image
in $U^{(2)}_\Delta$
(by (\ref{use3}))
and is continuous as a map to~$M_\tau$,
using Lemma~\ref{insUm} we find $\sigma\geq\tau$
such that\vspace{-.4mm}
$F^*_\gamma(\partial\Delta\times [0,1])\sub U_{\Delta,\sigma}^{(2)}$
and $F^*_\gamma|_{\partial\Delta
\times [0,1]}$
is continuous as a map to
$U_{\Delta,\sigma}^{(2)}$.
Thus
$\Xi_\Delta(\{\gamma|_{|\Sigma^*|}\} \times [0,1])\sub
C(\partial \Delta, V_{\Delta,\sigma}^{(2)})$.\vspace{-.4mm}
The\linebreak
mapping $\partial \Delta\times [0,1]\to V^{(2)}_{\Delta,\sigma}$,
$(x,t)\mto \phi_\Delta(F^*_\gamma(x,t))
=\Xi_\Delta(\gamma|_{|\Sigma^*|},t)(x)$
is\linebreak
continuous and has compact image $K\sub V^{(2)}_{\Delta,\sigma}$.
By Definition~\ref{defwlfl}\,(f),
there is $\beta\geq \sigma$ such that
$\conv_2(K)\sub V_{\Delta,\beta}$.
Now $\Phi_\Delta$, restricted to
$C(\partial\Delta, E_{\Delta,\sigma})$, is
continuous as a map to $C(\Delta,E_{\Delta,\sigma})_{c.o.}$
by Lemma~\ref{basicfill}\,(a)
(applied with $E_{\Delta,\sigma}$ rather than $E=E_\Delta$),
and hence also as a map to
$C(\Delta, E_{\Delta,\beta})_{c.o.}$.
Furthermore,
$\Phi_\Delta(\Xi_\Delta(\gamma|_{|\Sigma^*|},t))(x)\in \conv_2(K)\sub V_{\Delta,\beta}$
for each $t\in [0,1]$ and $x\in \Delta$, by choice of~$\beta$.
As a consequence,
$F_\gamma\colon \Delta\times [\frac{1}{2},1]\to M_\beta$,
$(x,t)\mto \phi^{-1}_{\Delta,\beta}
\big(\Phi_\Delta(\Xi_\Delta(\gamma|_{|\Sigma^*|},2t-1))(x)\big)$
is a continuous map to~$U_{\Delta,\beta}\sub M_\beta$.\\[2mm]
To tackle also the case $t\in [0,\frac{1}{2}]$,
note that we may assume that $U^{(4)}_\Delta$
has been obtained by applying the construction
from the proof of Lemma~\ref{core4}
to $\wb{\phi}:=\phi_\Delta$ and $W:=U^{(2)}_\Delta$.
Hence, we may assume that \mbox{the restriction of} $\phi_\Delta$ to some subset
of $U^{(2)}_\Delta$
is a well-filled chart with core $U^{(4)}_\Delta$,
and $V^{(4)}_\Delta:=
\phi_\Delta(U^{(4)}_\Delta)=\bigcap_{\theta\geq \alpha_0}V^{(4)}_{\Delta,\theta}$.
Since $\gamma(\Delta)\sub U^{(4)}_\Delta$
and $\gamma|_\Delta$ is a continuous
map to $M_\alpha$,
after increasing $\sigma$ (and $\beta$) we may assume
that $\gamma(\Delta)$ is a compact subset
of $U^{(4)}_{\Delta,\sigma}:=\phi_\Delta^{-1}(V^{(4)}_{\Delta,\sigma})$,
by Lemma~\ref{insUm}.
Thus $L:=\phi_\Delta(\gamma(\Delta))$
is a compact subset of $V^{(4)}_{\Delta,\sigma}$.
Since
$V^{(4)}_{\Delta,\sigma}\sub V^{(2)}_{\Delta,\sigma}$
(by the construction in the proof of Lemma~\ref{core4}),
after increasing $\beta$ if necessary we may assume that
$\conv_2(L)\sub V_{\Delta,\beta}$
(by Definition~\ref{defwlfl}\,(f)).
Moreover, $\conv_2(L)\sub \conv_2(V^{(4)}_\Delta)\sub \phi_\Delta(W)
=V^{(2)}_\Delta$.
Thus $\phi_\Delta^{-1}(\conv_2(L))$
is a compact subset of $M_\beta$
which is contained in $U^{(2)}_\Delta$.
In view of Lemma~\ref{insUm}, after increasing~$\beta$
we may assume that $\phi_\Delta^{-1}(\conv_2(L))$
is a compact subset of $U^{(2)}_{\Delta,\beta}$
and thus $\conv_2(L)\sub V^{(2)}_{\Delta,\beta}$.
Hence $\Phi_\Delta(\phi_\Delta\circ \gamma|_{\partial\Delta})\in
C(\Delta, \conv_2(L))\sub C(\Delta,V^{(2)}_{\Delta,\beta})$
(exploiting that $\Phi_\Delta$ takes $C(\partial\Delta, E_{\Delta,\beta})$
to $C(\Delta, E_{\Delta,\beta})$).
Since $\phi_\Delta$ is a well-filled chart, using
Definition~\ref{defwlfl}\,(f)
we see that $\conv_2(\conv_2(L))\sub V_{\Delta,\beta}$
may be assumed after increasing~$\beta$ further.
Thus $\conv_2(L,\conv_2(L))\sub V_{\Delta,\beta}$
in particular
and we obtain a continuous
map $\Delta \times [0,\frac{1}{2}]\to V_{\Delta,\beta}\sub E_{\Delta,\beta}$,
\[
(x,t)\mto (1-2t)\, \phi_\Delta(\gamma(x))+
2t\, \Phi_\Delta(\phi_\Delta\circ \gamma|_{\partial\Delta})(x)\,.
\]
Hence
$F_\gamma(x,t)=\Theta_\Delta(\gamma,x,t)=
\phi^{-1}_\Delta\big((1-2t)\,\phi_\Delta(\gamma(x))+2t\,
\Phi_\Delta(\phi_\Delta\circ \gamma|_{\partial\Delta})(x)\big)\in
U_{\Delta,\beta}$ holds,
and $F_\gamma\colon \Delta\times [0,\frac{1}{2}]\to 
U_{\Delta,\beta}$ is continuous.\vspace{1.3mm}

(g) Define $F^*_\gamma$ as in the proof of~(f).
If $\Delta\in \Sigma^*$ in the situation of~(g),
then $F_\gamma(x,t)=\gamma(x)=y$ for each $x\in \Delta$
if $t\in [0,\frac{1}{2}]$,
while
$F_\gamma(x,y)=F^*_\gamma(x,2t-1)=y$
if $t\in [\frac{1}{2},1]$,
by the inductive hypothesis.\\[2mm]
Now assume that $\rk(\Delta)=r$.
If $\Delta\sub \cE$, then $F_\gamma(t,x)=\gamma(x)=y$
for each $x\in \Delta$ and $t\in [0,1]$.
If $\Delta\not\sub \cE$ and $t\in [0,\frac{1}{2}]$,
given $x\in \Delta$
we have that $\phi_\Delta(F_\gamma(x,t))$
is a convex combination of the vectors
$\phi_\Delta(\gamma(x'))=\phi_\Delta(y)$
for several $x'\in \Delta$,
and thus $F_\gamma(x,t)=y$.
If $t\in [\frac{1}{2},1]$ and $x\in \Delta$,
then $F_\gamma(x,t)$
is the image under $\phi_\Delta^{-1}$
of a convex combination
of the vectors
$\phi_\Delta(F_\gamma^*(x',t))=\phi_\Delta(y)$
with $x'\in \partial\Delta$.
Since any such convex combination
is $\phi_\Delta(y)$,
it follows that
$F_\gamma(x,t)=y$.\vspace{1.3mm}

(h) Let $\gamma\in P$,
$x\in S$ and
$t\in [0,1]$.
Then $x\in |\Sigma^*|$.
Hence
$F_\gamma(x,t)=\gamma(x)$ if $t\in [0,\frac{1}{2}]$,
while
$F_\gamma(x,t)=F^*_\gamma(x,2t-1)=\gamma(x)$
if $t\in [\frac{1}{2},1]$,
by induction.\\[2.5mm]
Now take $x\in \cE$.
Then $x\in \Delta$ for some $\Delta\in \Sigma$
such that $\Delta\sub \cE$.
If $\rk(\Delta)=r$, then $F_\gamma(x,t)=\gamma(x)$
by definition of~$\Theta$.
If $\rk(\Delta)<r$,
then $\Delta\sub \Sigma^*$
and we see as in the case $x\in S$
that
$F_\gamma(x,t)=\gamma(x)$.
\end{proof}
For a single map
$\gamma_0\colon |\Sigma|\to M$,
we can deduce stronger conclusions.
\begin{la}[Individual Approximations]\label{pivo2}
In the setting
of~{\bf\ref{repeatset}},
let $\Sigma$ be a finite simplicial complex,
$\gamma_0\colon |\Sigma|\to M$
be a continuous function
and $Q\sub C(|\Sigma|,M)_{c.o.}$
be a neighbourhood of~$\gamma_0$.
Let $\cE\sub |\Sigma|$ be a subset such that
$\cE = \bigcup\, \{\Delta\in \Sigma\colon \Delta\sub \cE\}$.
Assume that there exists $\alpha\in A$ such that $\gamma_0(\cE)\sub M_\alpha$
and $\gamma_0|_\cE$ is continuous
as a map to~$M_\alpha$.
Then there exists $\beta\geq \alpha$
and a continuous
map $\eta\colon |\Sigma|\to M_\beta$
such that
$\eta|_\cE=\gamma_0|_\cE$
and $\eta\in Q$.
Moreover,
there exists a homotopy
$H\colon |\Sigma|\times [0,1]\to M$
relative~$\cE$ from $\gamma_0$ to~$\eta$,
such that $H(\sbull,t)\in Q$
for each $t\in [0,1]$.
\end{la}
\begin{proof}
Let $P$, $S$, $\Theta$ and further notation
be as in Lemma~\ref{pivotal}, applied to
the given data $\gamma_0$, $Q$ and $\cE$.
Abbreviate $X:=|\Sigma|$. We claim:\\[2.5mm]
\emph{There exists a homotopy $G\colon X \times [0,1]\to M$
relative $\cE$
from $\gamma_0$ to some $\gamma\colon X \to M$
such that $\gamma(x)\in M_\infty$ for all $x\in S$
and $G(\sbull,t)\in P$ for $t\in [0,1]$.}\\[2.5mm]
If this is true, then $\eta:=\eta_\gamma$ (from Lemma~\ref{pivotal}\,(c))
is a continuous map from~$X$ to
some~$M_\beta$, by Lemma~\ref{pivotal}\,(d).\footnote{Using that
$\gamma|_\cE=\gamma_0|_\cE$ is a continuous map
to~$M_\alpha$.}
Furthermore, the map
\[
H\colon X\times [0,1]\to M\, , \quad
H(x,t)\, := \, \Theta(G(\sbull,t),x,t)
\]
is a homotopy from~$\gamma_0$ to~$\eta$
(cf.\ Lemma~\ref{sammeltaxi}\,(c) for the continuity
of~$H$),
and in fact a homotopy relative~$\cE$,
because $G(\sbull,t)|_\cE=\gamma|_\cE$
and hence $\Theta(G(\sbull,t),x,t)=G(x,t)=\gamma(x)$
for each $x\in \cE$ and $t\in [0,1]$,
by Lemma~\ref{pivotal}\,(h).\\[2.5mm]
\emph{Proof of the claim}.
There exist $\ell \in \N$,
compact sets $K_1,\ldots, K_\ell\sub X$
and open subsets $W_1,\ldots, W_\ell\sub M$
such that
$\gamma_0\in \bigcap_{j=1}^\ell\lfloor K_j, W_j\rfloor\sub P$.
Given $x\in S$, let
$I_x$ be the set of all $j\in \{1,\ldots,\ell\}$
such that $x \in K_j$,
define $J_x:=\{1,\ldots,\ell\}\setminus I_x$,
and $G_x:=X
\setminus \bigcup_{j\in J_x}K_j$.
By Lemma~\ref{core4},
for each $x\in S$
there is
a well-filled chart $\phi_x\colon U_x\to V_x$
such that
$\gamma_0(x)\in U_x^{(2)}$
for some core $U_x^{(2)}$ of~$\phi_x$,
and $U_x\sub \bigcap_{j\in I_x}W_j$
(if $I_x=\emptyset$, we define the preceding
intersection as~$M$).
We choose $U_x^{(4)}\sub U^{(2)}_x$
with $\gamma_0(x)\in U_x^{(4)}$
as in Lemma~\ref{u4v4},
and set $V_x^{(4)}:=\phi_x(U^{(4)}_x)$. 
Pick a metric $d$ on $X$
defining its topology.
There exists $\ve >0$
such that the closed $d$-balls $\wb{B}_\ve(x)\sub X$
for $x\in S$
are pairwise disjoint, $\wb{B}_\ve(x)\sub G_x$,
and $\gamma_0(\wb{B}_\ve(x))\sub U_x^{(4)}$.
Set $T:=\{x\in S\colon \gamma_0(x)\not\in M_\infty\}$.
Since $\gamma_0(\cE) \sub M_\infty$,
we then have $T\sub \;X\setminus \cE$.
Hence, after shrinking~$\ve$ further if necessary, we may
assume that
\begin{equation}\label{hencrel}
\cE \cap \bigcup_{x\in T} \wb{B}_\ve(x)\;=\;
\emptyset\,.
\end{equation}
Given $x\in S$,
pick $v_x\in V_x^{(4)}\cap V_{x,\infty}$,
where $V_{x,\infty}=\bigcup_{\alpha\geq \alpha_0}V_{x,\alpha}$
is as in Definition~\ref{defwlfl}\,(e).
We define $G\colon X\times [0,1]\to M$
for $z\in X$ and $t\in [0,1]$
as follows:
If $z\in X \setminus
\bigcup_{x\in T} B_\ve(x)$,
we set
\[
G(z,t)\; :=\; \gamma_0(z)\, .
\]
If
$z\in \wb{B}_\ve(x)$ for some $x\in T$, we set
\[
G(z,t):=\phi_x^{-1}\left(
t\Big(\Big( 1-\frac{d(z,x)}{\ve}\Big)v_x
+ \frac{d(z,x)}{\ve}\phi_x(\gamma_0(z))\Big)
+(1-t)\phi_x(\gamma_0(z))\right)\,.
\]
Then $G$ is continuous,
and $\gamma:=G(\sbull,1)$ satisfies
$\gamma(x)\in M_\infty$ for all $x\in S$.
In fact: $\gamma(x)=\gamma_0(x)\in M_\infty$
if $x\in S\setminus T$,
while $\gamma(x)=\phi_x^{-1}(v_x)\in M_\infty$
if $x\in T$.
If $z\in \cE$ and $t\in [0,1]$,
then $G(z,t)=\gamma_0(z)$,
by (\ref{hencrel}) and definition of~$G$.
Hence~$G$ is a homotopy relative~$\cE$
from $\gamma_0$ to~$\gamma$.
Finally, we have $\zeta:=G(\sbull,t)\in P$
for each $t\in [0,1]$.
To see this, let $j\in \{1,\ldots, \ell\}$
and $z\in K_j$. If $z\in X \setminus\bigcup_{x\in T}
\wb{B}_\ve(x)$, then $\zeta(z)=\gamma_0(z)\in W_j$.
If, on the other hand, $z\in \wb{B}_\ve(x)$
for some $x\in T$, then
\begin{eqnarray*}
\phi_x(\zeta(z)) & \in &
\conv_2(\conv_2(V^{(4)}_x),V^{(4)}_x)
\; \sub \; \conv_2(V^{(2)}_x,V^{(4)}_x)\\
&\sub & \conv_2(V^{(2)}_x,V^{(2)}_x)\; \sub \; V_x
\end{eqnarray*}
and thus $\zeta(z)\in U_x\sub W_j$
(noting that $z\in \wb{B}_\ve(x) \sub G_x$ implies $j\in I_x$).
Thus $\zeta(K_j)\sub W_j$ for each $j\in \{1,\ldots,\ell\}$
and hence $\zeta\in P$, as required.
This completes
the proof of the claim and hence
also the proof of Lemma~\ref{pivo2}.
\end{proof}
\section{The main result and first consequences}\label{secmainprf}
We shall deduce Theorem~\ref{manthm}
from a more general theorem
dealing with sets $[(X,C),(M,p)]$ of
homotopy classes.
\begin{numba}
If $X$ and $Y$ are topological spaces,
$C \sub X$ a closed set
and $p\in Y$, let
\[
[(X,C),(Y,p)]
\]
be the set of all equivalence classes $[\gamma]$
of continuous mappings $\gamma\colon X\to Y$
such that $\gamma|_C=p$,
using homotopy relative~$C$
as the equivalence relation.
If also $Z$ is a topological space,
$q\in Z$ and $f\colon Y\to Z$ is a continuous map such that
$f(p)=q$, we obtain a map
\[
[(X,C),f]
\colon [(X,C),(Y,p)]\to
[(X,C),(Z,q)]\,,\quad [\gamma]\mto [f\circ\gamma]\,.
\]
We simply write $f_*:=[(X,C),f]$ if the meaning is clear
from the context. If $C=\emptyset$,
it is customary to write $[X,Y]$ instead of $[(X,C),(Y,p)]$.
\end{numba}
\begin{numba}\label{newsett}
Now let $F$ be a finite-dimensional vector space,
$\Sigma$ be a finite simplicial
complex of simplices in~$F$
and $X:=|\Sigma|\sub F$.
Let $C\sub X$ be a subset
which is a union of simplices, i.e.,
$C=\bigcup\{\Delta\in \Sigma\colon \Delta\sub C\}$.
\end{numba}
\begin{thm}\label{newthm}
Let $X=|\Sigma|$ and $C\sub X$ be as in {\bf\ref{newsett}},
$M$ a topological space
and $(M_\alpha)_{\alpha\in A}$ be a directed family of
topological spaces whose union
$M_\infty:=\bigcup_{\alpha\in A}M_\alpha$
is dense in~$M$.
Assume that all inclusion maps $\lambda_\alpha\colon M_\alpha\to M$
and $\lambda_{\beta,\alpha}\colon M_\alpha\to M_\beta$ $($for $\alpha\leq\beta)$
are continuous.
For $p\in M_\infty$, abbreviate $A_p:=\{\alpha\in A\colon p\in M_\alpha\}$.
If $M$ admits well-filled charts,
then
\[
[(X,C), (M,p)]\;=\; \dl_{\alpha\in A_p} \,[(X,C),(M_\alpha,p)]
\]
as a set, for each $p\in M_\infty$.
\end{thm}
\begin{proof}
We may assume that $p\in M_\alpha$
for each $\alpha\in A$.
The sets $[(X,C),(M,p)]$
form a direct system~$\cS$ of sets,
with the bonding maps
$(\lambda_{\beta,\alpha})_*:=[(X,C),\lambda_{\beta,\alpha}]$.
We let $D:=\dl\,[(X,C),(M_\alpha,p)]$\vspace{-.8mm}
be the direct limit in the category
of sets,
with limit maps $\mu_\alpha\colon [(X,C),(M_\alpha,p)]\to D$.
Since the maps $(\lambda_\alpha)_*:=[(X,C),\lambda_\alpha]\colon
[(X,C),(M_\alpha,p)]\to
[(X,C),(M,p)]$ form a cone over~$\cS$,
there is a unique map
$\psi \colon D \to [(X,C),(M,p)]$
with $\psi\circ \mu_\alpha=(\lambda_\alpha)_*$
for all $\alpha\in A$.\\[2.5mm]
\emph{$\psi$ is surjective}.
Let $[\gamma_0]\in [(X,C),(M,p)]$
be the equivalence class of a continuous map
$\gamma_0\colon X \to M$
with $\gamma_0|_C=p$.
Applying Lemma~\ref{pivo2}
with $Q:=C(X,M)$ and $\cE:=C$,
we obtain $\beta\in A$
and a homotopy
$H\colon X\times [0,1]\to M$
relative~$C$ from~$\gamma_0$ to
some continuous map $\eta\colon X\to M_\beta$.
Then $[\gamma_0]=[\eta]=(\lambda_\beta)_*([\eta])
=\psi(\mu_\beta([\eta]))\in \im(\psi)$.
Thus~$\psi$ is surjective.\\[2.7mm]
\emph{$\psi$ is injective}.
To see this,
let $g,h \in D$ with
$\psi(g)=\psi(h)$.
There is $\alpha\in A$
such that $g=\mu_\alpha([\sigma])$
and $h=\mu_\alpha([\tau])$
for certain $[\sigma],[\tau]\in [(X,C),(M_\alpha,p)]$
with continuous maps
$\sigma,\tau\colon X\to M_\alpha$.
Then $(\lambda_\alpha)_*([\sigma])=(\lambda_\alpha)_*([\tau])$,
whence there is a homotopy
$\gamma_0\colon X \times [0,1]\to M$
relative $C$
from $\sigma$ to $\tau$,
considered as maps to~$M$.
Choose a triangulation $\Sigma'$ of $X \times [0,1]\sub F\times \R$
such that
\[
C \times [0,1]\;=\;
\bigcup \big\{\Delta\in\Sigma'\colon
\Delta\sub C \times [0,1]\big\}
\]
and
\[
X\times \{0,1\} \;=\;
\bigcup \big\{\Delta\in\Sigma'\colon
\Delta\sub X \times \{0,1\} \big\}
\]
(this is always possible, by standard arguments).
Applying
Lemma~\ref{pivotal}\linebreak
to~$\Sigma'$, $\gamma_0$, $Q:=C(X\times [0,1],M)$
and
\[
\cE\; :=\;
(C \times [0,1])\cup
(X\times \{0,1\})\,,
\]
we obtain $\beta\geq \alpha$,
a continuous map $\eta\colon X\times [0,1]\to M_\beta$
and a homotopy
$H \colon (X\times [0,1])\times [0,1]\to M$
relative~$\cE$
from $\gamma_0$ to $\eta$.
Because~$H$ is a homotopy relative~$\cE$,
we have
\[
\eta(x,0)\, =\, \gamma_0(x,0)\, =\, \sigma(x)
\]
and $\eta(x,1)=\gamma_0(x,1)=\tau(x)$
for all $x\in X$, and furthermore
\[
\eta(x,t)\, =\, \gamma_0(x,t)\, =\, p\quad
\mbox{for all $x\in C$ and $t\in [0,1]$.}
\]
Hence $\eta$ is a homotopy relative~$C$
from $\sigma$ to $\tau$, considered as
maps to~$M_\beta$.
Consequently, $[\sigma]=[\tau]$ in $[(X,C),(M_\beta,p)]$
and thus $g=(\lambda_\beta)_*([\sigma])=(\lambda_\beta)_*([\tau])=h$.
\end{proof}
\begin{rem}
Theorem~\ref{newthm}
and its proof easily extend
to sets
$[(X,C),(M,P)]$ of homotopy classes of mappings
between space pairs,
where $X,C$ and~$M$ are as before
and $P\sub M$ is a subset
such that $P\sub M_\theta$
for some~$\theta \in A$
and both~$M$ and~$M_\theta$ induce
the same topology on~$P$.
\end{rem}
{\bf Proof of Theorem~\ref{manthm}.}
Let $D:=\dl\,\pi_k(M_\alpha,p)$\vspace{-.5mm}
and $\psi \colon D \to \pi_k(M,p)$
be as in {\bf\ref{reudlsit}}.
If $k\geq 1$ or if $M$ and each $M_\alpha$
is a topological group
and each $\lambda_\alpha$ and $\lambda_{\beta,\alpha}$
a homomorphism, then also $\psi$
is a homomorphism of groups.
Since $\psi$ is a bijection by Theorem~\ref{newthm}
(and hence an isomorphism of groups
in the cases just described),
Theorem~\ref{manthm} is established.\,\vspace{2.5mm}\Punkt

\noindent
We record another simple consequence.
It mainly is of interest if a manifold
$M$ is a directed union
of manifolds admitting weak direct limit charts.
\begin{cor}\label{complifi}
Let $M$ be a topological space
and $(M_\alpha)_{\alpha\in A}$ be a directed family of
topological spaces
such that $M=\bigcup_{\alpha\in A}M_\alpha$.
Assume that all inclusion maps $M_\alpha\to M$
and $M_\alpha\to M_\beta$ $($for $\alpha\leq\beta)$
are continuous,
and that $M$ admits well-filled charts.
Then the path components of~$M$
are the unions of those of the steps:
\begin{equation}\label{wteqlty}
M_{(p)}\;=\; \bigcup_{\alpha\in A_p}\, (M_\alpha)_{(p)}\quad
\mbox{for all $\,p\in M$.}
\end{equation}
\end{cor}
\begin{proof}
Let $p\in M$, say $p\in M_\alpha$.
It is clear that
$\bigcup_{\beta\geq \alpha}(M_\beta)_{(p)}\sub M_{(p)}$.
To prove the converse inclusion,
let $q\in M_{(p)}$. There exists $\beta\geq\alpha$
such that $q\in M_\beta$.
Since
\[
(\lambda_\alpha)_*((M_\alpha)_{(p)})
\;=\; M_{(p)}\;=\; 
M_{(q)}\;=\; (\lambda_\beta)_*((M_\beta)_{(q)})
\]
and $\pi_0(M)=\dl\;\pi_0(M_\gamma)$\vspace{-.7mm}
by Theorem~\ref{manthm} (applied with $k=0$),
there exists $\gamma\geq \alpha,\beta$
such that
\[
(\lambda_{\gamma,\alpha})_*((M_\alpha)_{(p)})\;=\;
(\lambda_{\gamma,\beta})_*((M_\beta)_{(q)})
\]
(see (\ref{bscdl2})),
where we use the natural mappings
$(\lambda_\alpha)_*\colon \pi_0(M_\alpha)\to \pi_0(M)$,\linebreak
$(\lambda_\beta)_*\colon \pi_0(M_\beta)\to \pi_0(M)$,
$(\lambda_{\gamma,\alpha})_*\colon \pi_0(M_\alpha)\to \pi_0(M_\gamma)$
as well as\linebreak
$(\lambda_{\gamma,\beta})_*\colon \pi_0(M_\beta)\to \pi_0(M_\gamma)$.
Thus $(M_\gamma)_{(p)}=(M_\gamma)_{(q)}$
and thus $q\in (M_\gamma)_{(p)}$,
entailing that equality holds in~(\ref{wteqlty}).
\end{proof}
If $M_\infty:=\bigcup_{\alpha\in A}M_\alpha$
is merely dense in~$M$, the same argument shows that
$(M_\infty)_{(p)}=\bigcup_{\alpha\in A_p}\, (M_\alpha)_{(p)}$
for each $p\in M_\infty$.
\section{When the inclusion map is a weak\\
\hspace*{.3mm}homotopy equivalence}\label{secpalais}
We now extend Palais' result recalled in the introduction:
under suitable hypotheses,
the inclusion map $M_\infty\to M$ is a weak homotopy equivalence.
\begin{prop}\label{genpalais}
Assume that $M$ admits well-filled charts
in the situation of Theorem~{\rm\ref{manthm}},
and that $\cO$ is a topology on $M_\infty$
with the following properties:
\begin{itemize}
\item[\rm(a)]
All of the inclusion maps
$\sigma_\alpha \colon M_\alpha\to (M_\infty,\cO)$ $($for $\alpha\in A)$
as well as
$\sigma \colon (M_\infty,\cO)\to M$ are continuous;
\item[\rm(b)]
$(M_\infty,\cO)=\bigcup_{\alpha\in A}M_\alpha$
is compactly retractive.
\end{itemize}
Then $\sigma$ is a weak homotopy equivalence.
\end{prop}
\begin{proof}
We shall re-use notation from {\bf\ref{reudlsit}}.
Let $k\in \N_0$ and $p\in M_\infty$;
equip $M_\infty$ with the topology~$\cO$.
We have to show that $\sigma_*\colon
\pi_k(M_\infty,p)\to \pi_k(M,p)$
is a bijection.\\[2.3mm]
\emph{$\sigma_*$ is surjective.}
If $g \in \pi_k(M,p)$,
then $g=(\lambda_\alpha)_*(h)$
for some $\alpha\in A_p$ and $h\in \pi_k(M_\alpha,p)$,
by Theorem~\ref{manthm}.
Since $\lambda_\alpha=\sigma\circ \sigma_\alpha$,
it follows that $g=\sigma_*((\sigma_\alpha)_*(h))$
is in the image of $\sigma_*$.\\[2.5mm]
\emph{$\sigma_*$ is injective.}
To see this, let $[\gamma_1],[\gamma_2]\in \pi_k(M_\infty,p)$
such that $\sigma_*([\gamma_1])=\sigma_*([\gamma_2])$.
By compact retractivity of
$M_\infty=\bigcup_{\alpha\in A}M_\alpha$,
there exists $\alpha\in A$ such that
both $\gamma_1$ and $\gamma_2$ have image in $M_\alpha$
and their corestrictions
$\eta_j:=\gamma_j|^{M_\alpha}$
are continuous for $j\in \{1,2\}$.
Then $[\gamma_j]=(\sigma_\alpha)_*([\eta_j])$
and hence
$\psi(\mu_\alpha([\eta_j]))
=(\lambda_\alpha)_*([\eta_j])
=\sigma_*((\sigma_\alpha)_*([\eta_j]))
=\sigma_*([\gamma_j])$,
implying $\psi(\mu_\alpha([\eta_1]))=
\psi(\mu_\alpha([\eta_2]))$.
Since $\psi$ is bijective,
it follows that $\mu_\alpha([\eta_1])=\mu_\alpha([\eta_2])$
and thus $(\lambda_{\beta,\alpha})_*([\eta_1])=
(\lambda_{\beta,\alpha})_*([\eta_2])$
for some $\beta\geq\alpha$ (see (\ref{bscdl2})).
Because $\sigma_\alpha=\sigma_\beta\circ \lambda_{\beta,\alpha}$
and hence
$[\gamma_j]=(\sigma_\alpha)_*([\eta_j])=
(\sigma_\beta)_*((\lambda_{\beta,\alpha})_*([\eta_j]))$
for $j\in \{1,2\}$,
we deduce that $[\gamma_1]=[\gamma_2]$.
\end{proof}
Now Corollary~\ref{corpal} (and slightly more)
readily follows.\\[2.7mm]
{\bf Proof of Corollary~\ref{corpal}.}
Let $A:=\cF$ be the set of all finite-dimensional
vector subspaces~$F$ of~$E_\infty$.
If $U$ is open, then $\phi:=\id_U\colon U\to U\sub E$
is a well-filled chart of~$U$ such that each $q\in U$
is contained in some core of~$\phi$
(see Example~\ref{newd}\,(ii)),
with $U_F:=V_F:=U\cap F$ and $\phi_F:=\id_{U_F}$.
If $U$ is semi-locally convex,
then each $q\in U$ has a convex relatively open
neighbourhood $W\sub U$. Then
$W\cap E_\infty$ is dense in~$W$,
and $\phi:=\id_W\colon W\to W\sub E$
is a well-filled chart such that $q$
is contained in some core of~$\phi$,
by Example~\ref{newDD}
(with $U_F:=V_F:=W\cap F$ and $\phi_F:=\id_{U_F}$).
We are therefore in the situation
of Theorem~\ref{manthm}.
Let $\cT$ be the topology~$\cO$ on~$U_\infty$ described
in Corollary~\ref{corpal}.
Or, more generally,
let $\cT$ be any topology on $U_\infty$ which is
coarser than the direct limit topology
on $\dl\,(U\cap F)$\vspace{-.7mm}
(where $U\cap F$ is equipped
with the topology induced by the finite-dimensional
vector subspace $F\sub E_\infty$)
but finer than the topology induced on~$U_\infty$
by the finest vector
topology on~$E_\infty$ (if $E$ is locally convex, one can also
use the finest locally convex vector\linebreak
topology
as a lower bound).\footnote{See~\cite{Bis}
and \cite{KaK} for the relations between these topologies.}
Then $U_\infty=\bigcup_{F\in \cF} (U\cap F)$
is compactly retractive
because so is $E_\infty=\bigcup_{F\in\cF}F$
with the finest locally convex topology (see,
e.g.,
\cite[Proposition~7.25\,(iv)]{HaM}).
Thus Proposition~\ref{genpalais}
applies: The inclusion map $(U_\infty, \cT)\to U$
is a weak homotopy equivalence.\,\vspace{2.5mm}\Punkt
\begin{rem}\label{givestrct}
Many criteria
for compact retractivity are known.
\begin{itemize}
\item[(a)]
For example, the direct limit topology
on the union $M=\bigcup_{n\in \N}M_n$
of an ascending sequence $M_1\sub M_2\sub\cdots$
of Hausdorff topological spaces
is compactly retractive if the direct sequence
is \emph{strict} in the sense
that each inclusion map $M_n\to M_{n+1}$
is a topological embedding
(e.g., by \cite[Lemma~1.7\,(d)]{FUN}
combined with \cite[Lemma~A.5]{NRW2};
cf.\ also \cite{Han}).
\end{itemize}
Further conditions (beyond strictness)
arise from the reduction to
modelling spaces performed in Proposition~\ref{inftesiml}\,(a).
On the level of locally convex spaces,
various criteria for compact retractivity are known.
One such criterion
was already encountered in preceding proof. Here are further ones:
\begin{itemize}
\item[(b)]
The locally convex direct limit topology
on $E=\bigcup_{n\in \N}E_n$
is compactly retractive for each
strict ascending sequence
$E_1\sub E_2\sub\cdots$
of complete locally convex topological
vector spaces (cf.\ Proposition~9~(i) and~(ii) in \cite[Ch.\,II, \S4, no.\,6]{BTV}
and Proposition~6 in \cite[Ch.\,III, \S1, no.\,4]{BTV}).
\item[(c)]
The locally convex direct limit topology
on $E=\bigcup_{n\in \N}E_n$
is compactly retractive for each
ascending sequence
$E_1\sub E_2\sub\cdots$
of Banach spaces, such that all inclusion maps
$E_n\to E_{n+1}$ are compact operators
(see Proposition~7 in \cite[Ch.\,III, \S1, no.\,4]{BTV},
or \cite{Flo}).
In this situation, $E$ is called
a \emph{Silva space} (or also a DFS-space).
\item[(d)]
For (LF)-spaces,
a quite concrete characterization of compact retractivity
is given in \cite[Theorem~6.4]{Wen}:
Let $E_1\sub E_2\sub \cdots$ be Fr\'{e}chet
spaces, with continuous linear inclusion maps.
Equip $E=\bigcup_{n\in \N}E_n$
with the locally convex direct limit topology.
Then $E=\bigcup_{n\in \N}E_n$
is compactly retractive
if and only if
for each $n\in \N$, there exists $m\geq n$
such that for all $k\geq m$,
there is a $0$-neighbourhood $U$ in $E_n$
on which $E_k$ and $E_m$ induce the
same topology.
In this case, $E$ is also regular and complete
\cite[Corollary to Theorem~6.4]{Wen}.
\end{itemize}
Further criteria
and references to the research literature
can be found in~\cite{Bie}.
\end{rem}
\section{Applications to typical Lie groups that are\\
\hspace*{.5mm}directed unions of Lie groups or manifolds}\label{seclie}
In this section, we show that our techniques
apply to all major classes of examples
of Lie groups~$G$
which are an ascending union $G=\bigcup_{n\in\N}G_n$
of Lie groups or manifolds $G_n$
(as compiled in~\cite{COM}).\\[2.5mm]
In Examples~\ref{exx1}--\ref{exx7},
we shall see that $G=\bigcup_{n\in \N}G_n$
has a weak direct limit chart and
$L(G)=\bigcup_{n\in \N}L(G_n)$ is compactly retractive,
whence $G=\bigcup_{n\in \N}G_n$ is compactly retractive
(by Proposition~\ref{inftesiml}\,(a)).
Hence Proposition~\ref{babcoreg}
gives information both concerning the homotopy groups
and the singular homology groups of~$G$.
In Example~\ref{exx8}, the same
reasoning applies to certain Lie groups~$G$
which can be written as a union $G=\bigcup_{n\in \N}M_n$
of Banach manifolds.
In Example~\ref{exx6}, compact retractivity can be violated,
but the group still has a direct limit chart
and thus Theorem~\ref{babythm}
provides information concerning the
homotopy groups.
\begin{example}\label{exx1} (\emph{Direct limits of finite-dimensional Lie groups}).
Consider an ascending sequence
$G_1\sub G_2\sub\cdots$
of finite-dimensional Lie groups,
such that the inclusion maps $G_n\to G_{n+1}$
are smooth homomorphisms.
Give $G=\bigcup_nG_n$ the Lie group structure
making it the direct limit Lie group
$\dl\, G_n$\vspace{-.7mm}
(see \cite[Theorem~4.3]{FUN},
or also \cite{NRW1} and \cite{DIR}
in special cases).
Then $G$ has a direct limit chart by construction
and $L(G)=\dl\,L(G_n)$\vspace{-1.2mm}
is compactly retractive
(see Remark~\ref{givestrct}\,(a) or~(b)).
\end{example}
\begin{example}\label{exdiffeo}
(\emph{Groups of compactly supported
diffeomorphisms}).
If $M$ is a $\sigma$-compact,
finite-dimensional smooth manifold,
there exists a sequence
$K_1\sub K_2\sub\cdots$
of compact subsets of~$M$ such that
$M=\bigcup_{n\in \N}K_n$
and $K_n\sub K_{n+1}^0$ (the interior
in $M$) for each $n\in \N$.
Then $(K_n)_{n\in \N}$ is a cofinal
subsequence of the directed set $\cK$ of all
compact subsets of~$M$.
Let $\Diff_c(M)$ be the Lie group
of all $C^\infty$-diffeomorphisms $\gamma\colon M\to M$
such that the closure of $\{x\in M\colon \gamma(x)\not=x\}$
(the support of~$\gamma$) is compact;
this Lie group is modelled on the LF-space $\cV_c(M)$
of compactly supported smooth vector fields
on~$M$.
Given $K\in \cK$, let $\Diff_K(M)$
be the Lie group of all $\gamma\in \Diff_c(M)$
supported in~$K$,
modelled on the Fr\'{e}chet space $\cV_K(M)$
of smooth vector fields supported in~$K$
(cf.\ \cite{Mic}, \cite{DIF} and \cite{GaN}
for the Lie group structures on these groups).
Then
\[
\Diff_c(M)\;=\;\bigcup_{K\in \cK} \Diff_K(M)
\]
and $\Diff_c(M)$ admits a direct limit chart (cf.\ \cite[\S5.1]{COM}).
Moreover, $\cV_c(M)=\bigcup_{K\in \cK}\cV_K(M)$
is compactly retractive
(see Remark~\ref{givestrct}\,(b)).
\end{example}
\begin{example}\label{exx3} (involving a mere weak direct limit chart).
We mention that $\Diff_c(M)$ can also be made
a Lie group modelled
on the space $\cV_c(M)$ of compactly
supported smooth vector fields, equipped
with the (usually properly coarser) topology
making it the projective limit
\[
\bigcap_{k\in \N_0}\cV_c^k(M)\;=\;
\pl_{k\in \N_0}\, \cV_c^k(M)
\]
of the LB-spaces
of compactly supported $C^k$-vector fields
(see \cite{DIF}, where this Lie group is denoted
$\Diff_c(M)\wt{\;}$).
Then $\Diff_c(M)\wt{\;}
= \bigcup_{K\in \cK} \Diff_K(M)$
and the chart of
$\Diff_c(M)\wt{\;}$ around $\id_M$ described in \cite[\S5.1]{COM}
is a weak direct limit chart
(albeit not a direct limit chart).
It is not hard to see
(with
Remark~\ref{givestrct}\,(b))
that $\pl_{k\in \N_0}\, \cV_c^k(M)=\bigcup_{K\in \cK}
\cV_K(M)$\vspace{-.1mm} is compactly\linebreak
retractive.
Hence
$\Diff_c(M)\wt{\;}=\bigcup_{K\in \cK} \Diff_K(M)$
is compactly retractive
(by Proposition~\ref{inftesiml}\,(a)).
\end{example}
\begin{example}\label{exx4} (\emph{Test function groups}).
Let $M$ and $\cK$ be as in Example~\ref{exdiffeo},
$H$ be a Lie group modelled
on a locally convex space,
and $r\in \N_0\cup\{\infty\}$.
Consider the ``test function group''
$C^r_c(M,H)$ of
$C^r$-maps $\gamma\colon M\to H$
such that the closure of $\{x\in M\colon \gamma(x)\not=1\}$
(the support of $\gamma$) is compact.
Given $K\in \cK$, let $C^r_K(M,H)$
be the subgroup of functions supported in~$K$.
Then $C^r_K(M,H)$ is a Lie group modelled
on $C^r_K(M,L(H))$, and
$C^r_c(M,H)$ is a Lie group
modelled on the locally convex direct
limit $C^r_c(M,L(H))=\dl\,C^r_K(M,L(H))$\vspace{-.7mm}
(see \cite{GCX}; cf.\ \cite{Alb}
for special cases, also \cite{NRW2}).
Then
\[
C^r_c(M,H)\;=\; {\textstyle \bigcup_{K\in \cK}}\, C^r_K(M,H)
\]
admits a direct limit chart (cf.\ \cite[\S7.1]{COM}).
Furthermore, $C^r_c(M,L(H))\!=$ $\bigcup_K\, C^r_K(M,L(H))$
is compactly retractive as a consequence
of Remark~\ref{givestrct}\,(b).
\end{example}
\begin{example}\label{exx5}
(\emph{Weak direct products of Lie groups}).
Given a sequence\linebreak
$(H_n)_{n\in \N}$ of Lie groups,
its weak direct product
$G:=\prod_{n\in \N}^*H_n$
is defined as the group of all
$(x_n)_{n\in \N}\in \prod_{n\in \N}H_n$
such that $x_n=1$ for all but finitely many~$n$;
it has
a natural Lie group structure~\cite[\S7]{MEA}.
Then $G=\bigcup_{n\in \N}G_n$, identifying the partial
product $G_n:=\prod_{k=1}^nH_k$ with a subgroup
of~$G$. By construction, $G=\bigcup_{n\in \N}G_n$
has a direct limit chart. Moreover,
$L(G)=\bigoplus_{n\in \N}L(H_n)
=\dl\, L(G_n)$\vspace{-.5mm}
is compactly retractive,
as locally convex direct sums are
regular \cite[Ch.\,3, \S1, no.\,4, Proposition~5]{BTV}
and induce the given topology on each finite
partial product
(cf.\ Propositions~7 or 8\,(i)
in \cite[Ch.\,2, \S4, no.\,5]{BTV}).
\end{example}
\begin{example}\label{exx7} (\emph{Lie groups of germs of analytic mappings}).
Let $H$ be a\linebreak
complex Banach-Lie group,
$\|.\|$ be a norm on $L(H)$ defining its topology,
$X$ be a complex
metrizable locally convex space
and $K\sub X$ be a non-empty compact set.
Then the set
$\Germ(K,H)$
of germs around~$K$ of $H$-valued complex analytic functions
on open neighbourhoods of~$K$ can be made a Lie group
modelled on the locally convex direct limit
\[
\Germ(K,L(H))\; =\; \dl\, \Hol_b(W_n,L(H))
\]
of the Banach spaces
$\cg_n:=\Hol_b(W_n,L(H))$ of bounded $L(H)$-valued
complex analytic functions on~$W_n$
(with the supremum norm),
where $W_1\supseteq W_2\supseteq \cdots$
is a fundamental sequence of open neighbourhoods of~$K$
in~$X$ such that each connected
component of $W_n$ meets~$K$  (see \cite{HOL}).
The group operation arises from
pointwise multiplication of representatives of germs.
The identity component $\Germ(K,H)_0$
is the union
\[
\Germ(K,H)_0\;=\; \bigcup_{n\in \N}G_n
\]
of the Banach-Lie
groups $G_n:=\langle [\exp_H\circ\, \gamma]\colon \gamma\in \cg_n\rangle$,
and $\Germ(K,H)_0=\bigcup_{n\in\N}G_n$
admits a direct limit chart~\cite[\S10.4]{COM}.
Moreover,
Wengenroth's result recalled in Remark~\ref{givestrct}\,(d)
implies that $\Germ(K,L(H))=\bigcup_{n\in \N}\cg_n$
is compactly retractive
(see \cite{DaG}),\footnote{If
$X$ and~$H$ are finite-dimensional
and $W_{n+1}$ is relatively compact in~$W_n$,
then the restriction maps $\Hol_b(W_n,L(H))\to \Hol_b(W_{n+1},L(H))$
are compact operators~\cite[\S10.5]{COM},
whence $\Germ(K,L(H))=\bigcup_{n\in \N}\cg_n$
is compactly retractive by the simpler Remark~\ref{givestrct}\,(c).}
and thus also $\Germ(K,H)_0=\bigcup_{n\in \N}G_n$.
\end{example}
\begin{example}\label{exx8}
(\emph{Lie groups of germs of analytic diffeomorphisms}).
If $X$ is a complex
Banach space
and $K\sub X$ a non-empty compact subset,
let $\GermDiff(K)$
be the set of germs around~$K$ of complex analytic diffeo\-morphisms
$\gamma\colon U\to V$
between open neighbourhoods $U$ and $V$ of~$K$
(which may depend on $\gamma$),
such that $\gamma|_K=\id_K$.
Then $\GermDiff(K)$
can be made a Lie group
modelled on the locally convex direct limit
\[
\Germ(K,X)_K\; :=\; \dl\, \Hol_b(W_n,X)_K\,,
\]
where the $W_n$ and $\Hol_b(W_n,X)$
are as in Example~\ref{exx7}
and $\Hol_b(W_n,X)_K:=\{\zeta \in \Hol_b(W_n,X)\colon \zeta|_K=0\}$
(see \cite[\S15]{COM} for the special case
$\dim(X)<\infty$, and \cite{Dah} for the general
result).
The group operation arises from
composition of representatives of germs.
Now the set $M_n$
of all elements of $\GermDiff(K)$
having a representative in $\Hol_b(W_n,X)_K$
is a Banach manifold, and
\[
\GermDiff(K)\;=\; \bigcup_{n\in \N}M_n
\]
has a direct limit chart (see \cite{Dah}; cf.\
\cite[Lemma~14.5 and \S15]{COM}).
Again, Wengenroth's characterization\footnote{Or simply
Remark~\ref{givestrct}\,(c), if $\dim(X)<\infty$.}
implies that $\Germ(K,X)_K=\bigcup_{n\in \N}\Hol_b(W_n,X)_K$
is compactly retractive
(see~\cite{Dah}), and hence also $\GermDiff(K)=\bigcup_{n\in \N}M_n$.
\end{example}
\begin{example}\label{exx6}
(\emph{Unit groups of ascending unions of Banach algebras}).
Let
$A_1\sub A_2\sub\cdots$ be
unital complex Banach algebras
(such that all inclusion maps
are continuous homomorphisms of unital algebras).
Give
$A:=\bigcup_{n\in \N}A_n$
the locally convex direct limit topology.
Then $A^\times$ is open in~$A$
and if $A$ is Hausdorff (which we assume now),
then
$A^\times$ is a complex Lie group\linebreak
\cite[Proposition~12.1]{COM}.
Moreover, $A^\times=\bigcup_{n\in \N}A_n^\times$
and the identity map $\id_{A^\times}$ is a direct
limit chart.\\[2.5mm]
If each inclusion map $A_n\to A_{n+1}$
is a topological embedding or each a compact
operator, then $A=\bigcup_{n\in \N}A_n$
and hence also $A^\times=\bigcup_{n\in \N}A_n^\times$
is compactly retractive (and thus Proposition~\ref{babcoreg} applies).
However, for more general
choices of the steps, $A=\bigcup_{n\in \N}A_n$
is not compactly retractive.\\[2.5mm]
To get an example for this pathology,
let $E_1\sub E_2\sub\cdots$ be a sequence
of\linebreak
Banach spaces whose locally convex direct
limit $E=\bigcup_{n\in \N}E_n$ is not regular\linebreak
(for example, a suitable ascending sequence
of weighted function spaces as in \cite[Remark~1.5]{BMS}).
Then $E=\bigcup_{n\in \N}E_n$ is not compactly retractive
(e.g., by Wengenroth's result
recalled in Remark~\ref{givestrct}\,(d)).
Consider $A_n:=\C\times E_n$
as a unital
complex Banach algebra with associative multiplication
$(z_1,x_1)\cdot (z_2,x_2):=(z_1z_2, z_1x_2+z_2x_1)$.
Since
$A := \dl\,A_n = \C\times \dl\,E_n = \C\times E$\vspace{-.5mm}
as a locally convex space,
$A=\bigcup_{n\in \N}A_n=\C\times (\bigcup_{n\in\N}E_n)$
is not compactly retractive
(nor is $A^\times=\bigcup_{n\in\N}A_n^\times$,
in view of Corollary~\ref{noncpreg}).
Of course,
the homotopy groups $\pi_k(A^\times)\isom
\pi_k(\C^\times)\times\pi_k(E)\isom \pi_k(\C^\times)$
(which are infinite cyclic if $k=1$
and trivial otherwise) can be calculated directly
in this example.
\end{example}
\section{Applications to typical Lie groups that\\
\hspace*{.5mm}contain a dense union of Lie groups}\label{secdnsungp}
We now describe typical examples of Lie groups
which contain a dense\linebreak
directed union
of Lie groups, and verify that
Theorem~\ref{manthm} applies.\\[2.5mm]
To test the applicability
of Theorem~\ref{manthm},
it is helpful to have a simple\linebreak
criterion for the existence
of well-filled charts.
The following lemma
serves this purpose.
It even applies to certain
topological groups.
\begin{la}\label{mkeazy}
Let $M$ be a topological group
that contains a directed union $M_\infty:=\bigcup_{\alpha\in A}M_\alpha$
of topological groups as a dense subset.
Assume that all inclusion maps $M_\alpha\to M$
and $M_\alpha\to M_\beta$ $($for $\alpha\leq\beta)$
are continuous homomorphisms.
If there exists a
well-filled chart $\phi\colon M\supseteq U\to V\sub E$
and a core $U^{(2)}$ of~$\phi$ such that $1\in U^{(2)}$,
then $M$ admits well-filled charts.
\end{la}
\begin{proof}
We re-use the notation from the introduction.
If $g\in M_\infty$, define
\[
\psi \colon gU\to V\,,\quad x\mto \phi(g^{-1}x)\,.
\]
After increasing $\alpha_0$,
we may assume that $g\in M_{\alpha_0}$.
Then $\psi$ is a well-filled chart with core
$gU^{(2)}$ (containing~$g$),
together with the charts $\psi_\alpha\colon gU_\alpha\to V_\alpha$,
$x\mto\phi_\alpha(g^{-1}x)$.
In fact,
the conditions~(a) and (b) from Definition~\ref{defweakdl}
hold by construction.
Since $gU\cap M_\infty=gU\cap gM_\infty=g(U\cap M_\infty)=gU_\infty=\bigcup_{\alpha\geq\alpha_0}
gU_\alpha$, condition~(d)
from Definition~\ref{defwlfl} holds.
Also (e) and (f) hold with $\psi(gU^{(2)})=V^{(2)}$,
as $V$ and $V_\alpha$
are unchanged and $\phi$ is a well-filled
chart.\\[2.5mm]
Now $M=M_\infty U^{(2)}$
by density of $M_\infty$ in~$M$
(cf.\ \cite[Lemma~3.17]{Str}).
Hence $M=\bigcup_{g\in M_\infty}gU^{(2)}$
is covered by cores of well-filled charts.
\end{proof}
We now prepare the discussion of weighted mapping groups.
If $(X,\|.\|)$ is a normed space, $Y$ a locally convex space,
$q$ a continuous seminorm on~$Y$
and
$p\colon X \to Y$
a continuous homogeneous polynomial,
we set
\begin{equation}\label{defnm1}
\|p\|_q \; :=\; \sup\{q(p(x))\colon x \in \wb{B}^X_1(0)\}\,.
\end{equation}
If $Y$ is a normed space and~$q$ its norm,
we also write $\|p\|:=\|p\|_q$.
\begin{numba}\label{situat}
Let $X=\R^d$, equipped with some norm,
$Y$
be a locally convex space,
$\Omega \sub X$ be open,
$r\in \N_0\cup\{\infty\}$
and $\cW$ be a set of smooth functions
$f\colon \Omega \to \R$
such that
the constant function~$1$
belongs to~$\cW$ and the following conditions
are satisfied:
\begin{itemize}
\item[(a)]
$f(x)\geq 0$
for all $f\in \cW$
and $x\in \Omega$;
\item[(b)]
For each $x\in \Omega$,
there exists
$f\in \cW$ such that $f(x)>0$;
\item[(c)]
For all $N\in \N$,
$f_1,\ldots, f_N\in \cW$
and $k_1,\ldots, k_N\in \N_0$
with \mbox{$k_1,\ldots, k_N\leq r$,}
there exist $C>0$
and $f\in \cW$ such that\footnote{Here $\delta^k_x f \colon X \to \R$
denotes the $k$-th Gateaux differential
of $f$ at~$x\in\Omega$, defined via
$\delta^k_x f (y):=\frac{d^k}{dt^k}\big|_{t=0} f(x+ty)$.}
\[
\|\delta^{k_1}_xf_1\| \cdot\ldots\cdot \|\delta^{k_N}_xf_N\| \;\leq\;
C\, f(x)\quad\mbox{for all $x\in \Omega$.}
\]
\end{itemize}
Let $C^r_\cW(\Omega,Y)$
be the set of all $C^r$-maps
$\gamma\colon \Omega\to Y$ such that
\[
\|\gamma\|_{f,k,q}\;:=\; \sup_{x\in \Omega}\, f(x)\, \|\delta^k_x\gamma \|_q \;<\;\infty
\]
for each $f\in \cW$, $k\in \N_0$ such that $k\leq r$,
and continuous seminorm~$q$ on~$Y$.
Let $C^r_\cW(\Omega,Y)^\bullet$ be the set of
all $\gamma\in C^r_\cW(\Omega,Y)$ such that
moreover
\[
f(x)\, \|\delta^k_x\gamma \|_q \to 0 \quad\mbox{as $\, x\to \infty$}
\]
in the Alexandroff compactification $\Omega \cup\{\infty \}$
of~$\Omega$.
Then $C^r_\cW(\Omega,Y)$
and $C^r_\cW(\Omega,Y)^\bullet$ are vector spaces
and the seminorms $\|.\|_{f,k,q}$ turn them into
locally convex spaces which are complete
if $Y$ is complete
(cf.\ \cite{Wa2}).
If $Q\sub Y$ is open, then
$C^r_\cW(\Omega, Q)^\bullet:=\{\gamma\in C^r_\cW(\Omega,Y)^\bullet\colon
\gamma(\Omega)\sub Q\}$ is open in
$C^r_\cW(\Omega, Y)^\bullet$.
\end{numba}
The conditions (a)-(c) imposed on $\cW$
imply a crucial property:
\begin{la}\label{isdnss}
$C^\infty_c(\Omega,Y)$ is dense
in $C^r_\cW(\Omega,Y)^\bullet$.
\end{la}
{\bf Proof} (sketch).
If $Y$ is finite-dimensional,
the assertion is immediate
from the scalar-valued
case treated in \cite[V.7\,a), p.\,224]{GWS}.
For the general case,
one first replaces $Y$ with a completion $\wt{Y}$
and reworks the proof of
\cite[V.7\,a), p.\,224]{GWS},
with minor modifications.\footnote{The completeness of $\wt{Y}$
ensures that the relevant vector-valued (weak) integrals
exist. As one continuous
seminorm~$q$ on~$Y$ suffices to describe a typical
neighbourhood of a given function in $C^r_\cW(\Omega,Y)^\bullet$,
the proof goes through
if we replace the absolute value~$|.|$~by~$q$.}
Then, in the last line of \cite[p.\,226]{GWS},
one replaces $(T_{m_1,m_2}f)(x_i^{(m_4)})\in \wt{Y}$
by a nearby element in~$Y$.\,\vspace{1mm}\Punkt
\begin{example}\label{decreas}
(\emph{Groups of rapidly
decreasing Lie group-valued maps}).
Given a Lie group $H$,
let $C^r_\cW(\Omega,H)^\bullet$
be the set of all
$C^r$-maps $\gamma\colon \Omega \to H$
for which there exists a chart $\kappa\colon P\to Q\sub L(H)$
of~$H$ around~$1$ with $\kappa(1)=0$,
a compact set $K\sub \Omega$
such that $\gamma(\Omega \setminus K)\sub P$,
and a compactly supported, smooth function
$h\colon \Omega\to [0,1]$ taking the constant value $1$
on some neighbourhood of $K$, such that
$(1-h)\cdot (\kappa\circ \gamma|_{\Omega\setminus K})
\in C^r_\cW(\Omega\setminus K,L(H))^\bullet$
(cf.\ \cite[Definition~6.4.1]{Wa2},
where this set is denoted by
$C^r_\cW(\Omega,H)^\bullet_{\text{max}}$).
Define
\[
C^r_\cW(\Omega,P)^\bullet \; := \;
\{\gamma\in C^r_\cW(\Omega, H)^\bullet\colon \gamma(\Omega)\sub P\}\, .
\]
Then $C^r_\cW(\Omega,H)^\bullet$ can be made a Lie group
modelled on $C^r_\cW(\Omega,L(H))^\bullet$
in a natural way, such
that, for some chart $\kappa\colon P\to Q$
as just described,
\[
\phi\colon U:=C^r_\cW(\Omega,P)^\bullet
\to C^r_\cW(\Omega,Q)^\bullet=:V\,,\quad
\gamma\mto \kappa \circ \gamma
\]
is a chart of $C^r_\cW(\Omega,H)^\bullet$
around~$1$
(see \cite{Wa2}; cf.\ \cite{BCR} for special
cases).\footnote{The Lie groups $\cS(\R^d,H;\cW)$
constructed in \cite{BCR}
coincide with the groups $C^\infty_\cW(\R^d,H)^\bullet$
by \cite[Lemma 6.4.23]{Wa2}.}\\[2.5mm]
To get some information on the homotopy groups
of $C^r_\cW(\Omega,H)^\bullet$,
let $\cK$ be the set of compact subsets of~$\Omega$,
directed by inclusion.
In~\cite{SMO}, it is shown that
$C^\infty_c(\Omega,H)=\bigcup_{K\in \cK}C^\infty_K(\Omega,H)$
is dense
in $C^r_\cW(\Omega,H)^\bullet$.
The restriction of~$\phi$ to the map
\[
\phi_K:=C^\infty_K(\Omega,\kappa)\colon
C^\infty_K(\Omega,P)\to C^\infty_K(\Omega,Q)
\]
from an open subset of $C^\infty_K(\Omega,H)$
to an open subset of $C^\infty_K(\Omega,L(H))$
is a chart of $C^\infty_K(\Omega,H)$
(see \cite{GaN}; cf.\ \cite[\S3]{GCX}).
Since $C^\infty_K(\Omega,P)=U\cap C^\infty_K(\Omega,H)$
and $C^\infty_K(\Omega,Q)=V\cap C^\infty_K(\Omega,L(H))$,
we are in the situation of Example~\ref{newd}\,(ii).
Thus $\phi$
is a well-filled chart
admitting cores around each $\gamma\in U$,
notably around $1\in U$.
Hence $C^r_\cW(\Omega,H)^\bullet$
admits well-filled charts (by Lemma~\ref{mkeazy})
and thus
\begin{equation}\label{tbimpr}
\pi_k(C^r_\cW(\Omega,H)^\bullet)=\dl\, \pi_k(C^\infty_K(\Omega,H))
=\pi_k(C^\infty_c(\Omega,H))\, ,
\end{equation}
using Theorem~\ref{manthm}
for the first equality
and Example~\ref{exx4} for the second.
\end{example}
If $\Omega=X$, then
the homotopy groups can be calculated more
explicitly.
\begin{thm}\label{thmD}
If $\Omega=X=\R^d$ in the preceding situation,
then
\[
\pi_k (C^r_\cW(\R^d,H)^\bullet)\;
\isom \; \pi_{k+d}(H)\quad\mbox{for all $k\in \N_0$.}
\]
\end{thm}
\begin{proof}
Let $\bS_d\sub \R^{d+1}$ be the $d$-dimensional sphere,
$*\in \bS_d$ be a point
and $C_*(\bS_d,H)$ be the group of $H$-valued continuous
maps on $\bS_d$ taking~$*$ to~$1$ (equipped with the
topology of uniform convergence).
Then
\begin{eqnarray*}
\pi_k(C^r_\cW(\R^d,H)^\bullet)
&\isom & \pi_k(C^\infty_c(\R^d,H))
\; \isom \; \pi_k (C_0(\R^d,H))\\
&\isom &
\pi_k(C_*(\bS_d,H))\;\isom
\;\pi_{k+d}(H)\,,
\end{eqnarray*}
using (\ref{tbimpr})
for the first isomorphism,
\cite[Theorem~A.10]{NeG}
for the second, and standard facts from
homotopy theory for the last.
\end{proof}
\begin{rem}\label{solveBCR}
Define $f_m\colon \R^d \to \R$
via $f_m(x):=(1+\|x\|^2)^m$
for $m\in \N_0$ (where $\|.\|$ is a euclidean
norm on $\R^d$),
and $\cW:=\{f_m\colon m\in \N_0\}$.
Then $C^\infty_\cW(\R^d,L(H))^\bullet$
is the Schwartz space $\cS(\R^d,L(H))$ of rapidly decreasing
smooth $L(H)$-valued maps on $\R^d$.
As a special case of Theorem~\ref{thmD},
the group $\cS(\R^d,H):=C^\infty_\cW(\R^d,H)^\bullet$
satisfies
\[
\pi_k(\cS(\R^d,H))\;=\; \pi_{k+d}(H)\quad \mbox{for all $\, k\in \N_0$.}
\]
This had been conjectured in \cite[p.\,130]{BCR},
and was open since 1981.
\end{rem}
\begin{example} (\emph{Weighted diffeomorphism groups}).
Let
$\Omega =X=Y$ and $r:=\infty$
in {\bf\ref{situat}},
and let
$\Diff_\cW(X)^\bullet$
be the set of all
$C^\infty$-diffeomorphisms
$\gamma\colon X\to X$
with $\gamma-\id_X\in C^\infty_\cW(X,X)^\bullet$
and $\gamma^{-1}-\id_X\in C^\infty_\cW(X,X)^\bullet$
(in \cite{Wa2}, this set is denoted $\Diff_\cW(X)^0$).
Then
\[
V\; :=\; \{\gamma\in C^\infty_\cW(X,X)^\bullet \colon \gamma+\id_X\in
\Diff_\cW(X)^\bullet\}
\]
is open in $C^\infty_\cW(X,X)^\bullet$
and $\phi\colon U:=\Diff_\cW(X)^\bullet\to V$, $\gamma\mto\gamma-\id_X$
a global chart for $\Diff_\cW(X)^\bullet$,
making it a Lie group (see \cite{Wa2}).\footnote{We
mention that special cases of such groups have been used
by physicists~\cite{Gol}. The weighted
diffeomorphism group of~$\R$ modelled on $\cS(\R,\R)$
has also been treated in~\cite{Mi2}.}
Because $C^\infty_c(X,X)$
is dense in $C^\infty_\cW(X,X)^\bullet$,
it follows that $\Diff_c(X)=\bigcup_{K\in \cK}\Diff_K(X)$ is dense
in $\Diff_\cW(X)^\bullet$,
where $\cK$ is the set of compact
subsets of~$X$ and $\Diff_c(X)$ as well as $\Diff_K(X)$
are as in Example~\ref{exdiffeo}.
Since,
for each $K\in \cK$,
the restriction of $\phi$ to a map
\[
\Diff_K(X)\to V\cap C^\infty_K(X,X)
\]
is a chart of $\Diff_K(X)$,
we are in the situation of
Example~\ref{newd}\,(ii)
and thus
$\pi_k(\Diff_\cW(X)^\bullet)=\dl_{K\in\cK}\,\pi_k(\Diff_K(X))$
for each $k\in \N_0$, by Theorem~\ref{manthm}.
\end{example}
\begin{rem}
We mention that
(unlike Example~\ref{decreas})
the preceding\linebreak
example
can also be deduced from Palais' classical theorem.
To this end,
let $\cF$ be the set of finite-dimensional vector subspaces
of $C^\infty_c(X,X)$, and $V_\infty:=V\cap C^\infty_c(X,X)$.
Because
$C^\infty_c(X,X)$ is dense in $C^\infty_\cW(X,X)^\bullet$,
using Palais' Theorem twice we see that
\begin{eqnarray*}
\pi_k (\Diff_\cW(X)^\bullet) &\isom&
\pi_k (V)
\, \isom \, \dl_{F\in \cF} \pi_k (V \cap F)
\, = \,  \dl_{F\in \cF} \pi_k (V_\infty \cap F)\\
&\isom & \pi_k (V_\infty)
\, \isom \, \pi_k (\Diff_c(X))\, .
\end{eqnarray*}
Hence
$\pi_k (\Diff_\cW(X)^\bullet)=\dl\, \pi_k (\Diff_K(X))$
(see Example~\ref{exdiffeo}).
Notably,
the inclusion map
$\Diff_c(X)\to \Diff_\cW(X)^\bullet$
is a weak homotopy equivalence.
\end{rem}
\begin{example}\label{finEXX}
Let $M$ be a $\sigma$-compact,
finite-dimensional smooth manifold, $r\in \N_0$
and $H$ be a Lie group. Then the inclusion map
\[
C^\infty_c(M,H)\to C^r_c(M,H)
\]
is a weak homotopy equivalence.\\[2.5mm]
To see this, let $\cK$ be the set of compact subsets of~$M$.
By~\cite{SMO},
$C^\infty_c(M,H)=\bigcup_{K\in \cK}C^\infty_K(M,H)$
is dense in $C^r_c(M,H)$.
Let $\kappa\colon P\to Q$ be a chart of~$H$
around~$1$ such that $P=P^{-1}$,
$\kappa(1)=0$ and $\kappa$ extends to a chart
with domain~$R$, such that $PP\sub R$.
Then
\[
\phi\; :=\; C^r_c(M,\kappa)\colon C^r_c(M,P)\to C^r_c(M,Q)\,,\quad
\gamma\mto\kappa\circ\gamma
\]
is a chart of $C^r_c(M,H)$
and
\[
\phi_K\; :=\; C^\infty_K(M,\kappa)\colon C^\infty_K(M,P)\to C^\infty_K(M,Q) \quad
\]
is a chart of $C^\infty_K(M,H)$,
for each compact subset $K\sub M$ (see \cite{GCX}).
It is clear that all conditions described
in Example~\ref{newd}\,(ii) are satisfied,
and thus~$\phi$ is a well-filled
chart admitting a core around~$1$.
Hence $C^r_c(M,H)$ admits well-filled
charts (by Lemma~\ref{mkeazy}), and thus
\[
\pi_k(C^r_c(M,H))\,\isom\,
\dl\,\pi_k(C^\infty_K(M,H))\,\isom\,
\pi_k(C^\infty_c(M,H))
\;\; \mbox{for each $k\in \N_0$,}
\]
by Theorem~\ref{manthm} and Example~\ref{exx4}.
The assertion follows.
\end{example}
\emph{Acknowledgement.}
The author is grateful to
K.-D. Bierstedt ($\dag$)
and\linebreak
S.-A. Wegner (Paderborn) for
advice on regularity properties
of (LF)-spaces. He thanks
A. Alldridge (Paderborn)
for the suggestion to
generalize Theorem~\ref{manthm}
to Theorem~\ref{newthm}.
The research was supported by
the German Research Foundation (DFG),
projects GL 357/4-1, GL 357/5-1
and GL~357/7-1.
{\footnotesize{\bf Helge Gl\"{o}ckner},
Universit\"{a}t Paderborn,
Institut f\"{u}r Mathematik,
Warburger Str.\ 100,
33098 Paderborn, Germany. E-Mail:
{\tt glockner@math.uni-paderborn.de}}
\end{document}